\documentclass[11pt, a4paper]{amsart}

\usepackage{mathtools, verbatim, bbm,cite}
\mathtoolsset{showonlyrefs}

\newtheorem{proposition}{Proposition}[section]
\newtheorem{lemma}[proposition]{Lemma}
\newtheorem{theorem}[proposition]{Theorem}
\newtheorem{corollary}[proposition]{Corollary}
\newtheorem{definition}[proposition]{Definition}

\newtheorem{remark}[proposition]{Remark}

\newtheorem{observation}[proposition]{Fact}

\renewcommand{\d}{\,\mathrm{d}}

\usepackage{xcolor}

\newcommand{\supp}{\mbox{supp}}
\newcommand{\dist}{\mbox{dist}}
\newcommand{\uno}{\operatorname{\mathbbm{1}}}

\newcommand{\hn}{\mathcal H^{n-1}}
\newcommand{\hm}{\mathcal H^{n-2}}

\newcommand{\cm}{\zeta}
\newcommand{\rp}{\mathbb R_{\scriptscriptstyle +}}

\newcommand{\sn}{{S^{n-1}}}
\newcommand{\sm}{S^{n-2}}

\newcommand{\on}{\operatorname{O}(n)}
\newcommand{\om}{\operatorname{O}(n-1)}
\renewcommand{\d}{\,\mathrm{d}}
\newcommand{\pole}{e}
\newcommand{\ra}{\mathbin{\tilde +}}

\newcommand{\oZ}{\operatorname{Z}}
\newcommand{\oV}{\operatorname{V}}

\newcommand{\oW}{\operatorname{W}}

\newcommand{\starn}{\mathcal S_{\scriptscriptstyle 0}^n}
\newcommand{\measn}{M(\sn )_{\scriptscriptstyle +}}

\newcommand{\cfn}{C(\sn )_{\scriptscriptstyle +}}
\newcommand{\bfn}{B(\sn )_{\scriptscriptstyle +}}
\newcommand{\R}{\mathbb R}
\newcommand{\rn}{\mathbb R^n}
\newcommand{\poly}{\mathcal P_{\mathbb Q}}
\newcommand{\somename}{strong Cara\-th\'eo\-dory and uniformly controlled }
\newcommand{\radon}{\operatorname{\overline R}}
\newcommand{\rhot}{\tau}

\newcommand{\G}{\xi}

\title{Measure-valued valuations on star bodies}

\author[J.~S. Ib\'a\~nez-Marcos]{Jorge S. Ib\'a\~nez-Marcos}
\address{Departamento de An\'alisis Matem\'atico y Matem\'atica Aplicada\\
Facultad de Matem\'aticas \\ Universidad Complutense de Madrid \\
Madrid 28040}
\email{jorgesib@ucm.es}

\author[M. Ludwig]{Monika Ludwig}
\address{Institut f\"ur Diskrete Mathematik und Geometrie,
Technische Universit\"at Wien,
Wiedner Hauptstra\ss e 8-10/1046,
1040 Wien, Austria}
\email{monika.ludwig@tuwien.ac.at}

\author[P. Tradacete]{Pedro Tradacete}
\address{Instituto de Ciencias Matem\'aticas (CSIC-UAM-UC3M-UCM)\\
Consejo Superior de Investigaciones Cient\'ificas\\
C/ Nicol\'as Cabrera, 13--15, Campus de Cantoblanco UAM\\
28049 Madrid, Spain.}
\email{pedro.tradacete@icmat.es}

\author[I. Villanueva]{Ignacio Villanueva}
\address{Departamento de An\'alisis Matem\'atico \\
Facultad de Matem\'aticas \\ Universidad Complutense de Madrid \\
Madrid 28040}
\email{ignaciov@mat.ucm.es}

\thanks{Research partially supported by grants PID2020-116398GB-I00, PID2024-162214NB-I00 and CEX2023-001347-S funded by MCIN/AEI/10.13039/501100011033. The first author has also received financial support from Ministerio de Ciencia, Innovación y Universidades through an FPU Grant FPU22/02055. The second author was funded in part by the Austrian Science Fund (FWF) DOI:10.55776/P34446.}

\begin{document}

\begin{abstract}
A complete classification of weak$^*$~continuous, measure-valued valuations is established on star bodies in  $\R^n$. Consequences are an integral representation of rotation equivariant, measure-valued valuations and a characterization of dual area measures.
\end{abstract}

\subjclass[2020]{52B45, 52A30, 28A33} 

\keywords{Star bodies; measure-valued valuation; dual area measure}

\maketitle

\section{Introduction} 
In his celebrated work, Hadwiger  \cite{Hadwiger} classified continuous, translation and rotation invariant, scalar-valued valuations on convex bodies in $\mathbb{R}^n$ and obtained a characterization of intrinsic volumes. Schneider  \cite{Schneider:75} extended Hadwiger's results to measure-valued valuations. He established a classification of weak$^*$~continuous, locally determined, translation invariant, and rotation equi\-variant valuations on convex bodies in $\R^n$ with values in the space of signed Radon measures on the unit sphere and a characterization of area measures 
(see  \cite{Schneider} for comprehensive information on intrinsic volumes, area  
measures, and valuations on convex bodies).

Dual intrinsic volumes were introduced by Lutwak \cite{Lutwak75} within the framework of the dual Brunn--Minkowski theory. They are defined on the space $\starn$ of star bodies in $\R^n$, equipped with the topology induced by the radial distance (see Section~\ref{sectionnotation} for precise definitions). 
A map $\oZ$ defined on $\starn$ with values in an additive abelian semigroup is a \emph{valuation}
if 
\begin{equation}\label{eq:val}
\oZ(K\cup L)+\oZ(K\cap L)=\oZ(K)+ \oZ(L).
\end{equation}
for every $K, L\in\starn$. In  \cite{TrViJMAA, TrViIntRep, Vi}, extending results by Klain~\cite{Klain1996, Klain1997}, a classification of continuous, scalar-valued valuations on $\starn$ was established, and a characterization of dual intrinsic volumes was obtained  (see \cite{Gardner, Schneider} for more information on the dual Brunn--Minkowski theory and its applications).

The aim of this paper is to establish classification results for measure-valued valuations on $\starn$. Let $\measn$ denote the space of non-negative Radon measures on the unit sphere $\sn$ and $\rp$ the non-negative real numbers. Our main classification result is given in the following theorem.   
\goodbreak

\begin{theorem}\label{th:rep}
A map $\oZ\colon  \starn \to \measn$ is a weak$^{\,*}\!$ continuous valuation satisfying $\oZ(\{0\})=0$ if and only if there exist 
$\mu\in \measn$ and a $\mu$-\somename map $U\colon \rp \times \sn  \to \measn$ such that $U(0,s)=0$ for $\mu$-almost every $s$ and 
$$\langle \oZ(L),g\rangle  =\int_{\sn } \langle U(\rho_L(s), s),g\rangle  \d\mu(s)$$
for every $L\in \starn$ and $g\in C(\sn)$. 
\end{theorem}

\noindent
Here, $\rho_L$ is the radial function of $L$ (see Subsection \ref{ss:star} for the definition). We say that a map
$U\colon \rp \times \sn  \to \measn$ is \emph{$\mu$-strong Carath\'eodory} if the map $(\lambda, s)\mapsto \langle U(\lambda, s),g\rangle$ is strong Carath\'eodory (see Subsection \ref{ss:SC}) with respect to $\mu$ for every $g\in C(\sn)$ (with the pairing  defined in \eqref{eq:pairing}).
We call the map $U$ \emph{uniformly controlled} if  
there is a function $ \cm\colon \rp \to \rp$ such that
$$\vert \langle U(\lambda, s), g\rangle\vert \leq \|g\|_{\scriptscriptstyle \infty}\, \cm(\lambda)$$
for every $\lambda\in\rp$, every $s\in\sn$, and every $g\in C(\sn )$.  
\goodbreak

Valuations that intertwine rotations are particularly interesting. We call a valuation  
$\oZ\colon \starn \to \measn$ \emph{rotation equivariant} if 
$$\langle \oZ(\phi L),\phi g\rangle = \langle \oZ(L),g\rangle $$ 
for all $L\in\starn$, $g\in C(\sn)$, and $\phi\in\on$, the orthogonal group in $\rn$. 

As a consequence of Theorem \ref{th:rep}, we will  
prove the following  
classification of rotation equivariant valuations.
Let $\mathcal H^k$ denote $k$-dimensional Hausdorff measure.

\begin{theorem}\label{th:rep equivariant}
A map $\,\oZ\colon  \starn\to \measn$ is a weak$^{\,*}\!$ continuous and rotation equivariant valuation satisfying $\oZ(\{0\})=0 $ if and only if there exists a weak$^{\,*}\!$ continuous function $\rhot\colon  \rp \to M([-1,1])_{\scriptscriptstyle +}$ such that $\rhot_0=0$ and
\begin{align*}
\langle \oZ (L),g\rangle= \int_\sn \int_{[-1,1]}\radon g(s,\alpha)\d\rhot_{\rho_L(s)}(\alpha)\d\hn(s) 
\end{align*}
for every $L\in\starn$ and $g\in C(\sn)$. 
\end{theorem}

\noindent
Here, $\radon g(s,\alpha)$ is the average of $g$ in the set $\{r\in\sn:\langle r,s\rangle=\alpha\}$ (see Subsection \ref{ss_rev}) and we also write $\rhot_\lambda$ for the measure $\rhot(\lambda)$.

Huang, Lutwak, Yang, and Zhang \cite{HLYZ} introduced the $q$-th dual area measure  $\tilde S_q(L)$ for $L\in\starn$  and $q\in\R$  within the dual Brunn--Minkowski theory. As a consequence of Theorem \ref{th:rep equivariant}, we obtain the following characterization.

\begin{theorem}\label{t:char}
A map $\oZ\colon \starn\to \measn$ is a weak$^{\,*}\!$ continuous, locally determined, rotation equi\-variant, $q$-homogeneous valuation if and only if there is $c\ge 0$ such that 
$\oZ=c\,\tilde{S_q}$. 
\end{theorem}

\noindent
Here, the map $\oZ$ is $q$-homogeneous if $\oZ(\lambda L)=\lambda^q\oZ(L)$ for every $\lambda>0$ and $L\in\starn$.
See Section \ref{ss:ld} for the definition of dual area measures and locally determined valuations on $\starn$. 

\goodbreak
We will establish our results in the setting of valuations on function spaces and identify $\starn$ with $\cfn$, the space of non-negative, continuous functions on $\sn$. We remark that the first classification result for valuations on function spaces was established for Sobolev functions \cite{ML2012}, followed by results on Lebesgue spaces, spaces of continuous functions, functions of bounded variation, and convex functions  (see \cite{Alesker, CLM2017, CLM2019, CLM2020, CLM2022, CLM2023, CLM2024, CLM2024b, CPTV2020, CPTV2021, Knoerr2024, Knoerr2024b, KnoerrUlivelli, LiL, LiMa, ML2011, mouamine2025klain, Mussnig2019, Mussnig2021, Mussnig2021b, TrViJMAA, TrViIntRep, TrViLattices, Tsang, Vi, Wang2014} for some recent results and the survey \cite{ML2023}). Measure-valued valuations were considered on convex functions on $\R^n$ in \cite{Knoerr2024b, LiL}. Our results establish a complete classification of measure-valued valuations that are weak$^*$~continuous on non-negative, continuous functions on $\sn$.  

The rest of the paper is structured as follows: In Section \ref{sectionnotation}, we present the notation and basic definitions used throughout the paper. In Section \ref{s:vvv}, we consider valuations on star bodies taking values in dual Banach spaces. Within this general framework and building on the scalar-valued case proven in \cite{TrViIntRep}, we derive some basic results and introduce the notion of a control measure for such valuations. In Section \ref{S:control}, given a weak$^*$ continuous valuation taking values in the space of measures $\measn$, we construct a control measure $\mu$. Proceeding as in  \cite{TrViIntRep}, we establish an integral representation for such a valuation with respect to the control measure and study some of its properties. Section \ref{S:Usection} is devoted to the proof of Theorem \ref{th:rep}. In Section \ref{S:equivariant}, we focus on valuations satisfying additional properties such as rotation equivariance, homogeneity, and local determination. There, we prove Theorems \ref{th:rep equivariant} and \ref{t:char}. In addition, we describe a valuation that is rotation equivariant and weak$^*$~continuous but not norm continuous. Finally, in the Appendix \ref{A:Inspection}, we include, for the reader's convenience, the details of an argument which is implicit in \cite{TrViIntRep}.

\section{Notation and Preliminaries}\label{sectionnotation}

Let $\sn$ be the unit sphere of the Euclidean space $\rn$, and let $C(\sn )$  denote the Banach space of continuous functions on $\sn$, endowed with the supremum norm 
$$\|f\|_{\scriptscriptstyle \infty}=\sup\{|f(s)|: {s\in \sn }\}.$$ 
We will denote the cone of non-negative functions in $C(\sn )$ by $\cfn$. Let $\Sigma_n$ be the Borel $\sigma$-algebra in $\sn $ and $B(\sn )$ the Banach space of bounded Borel measurable functions $g\colon\sn \to \R$, endowed with the supremum norm. We write $\bfn$ for the cone of non-negative functions in $B(\sn )$. Note that the formal inclusion of  $C(\sn )$ into $B(\sn )$ is a linear isometric embedding.  

\subsection{Star bodies}\label{ss:star} 
Given $x \in \rn$, let $[0,x]=\{\lambda x:0\leq \lambda\leq 1\}$ denote the line segment joining the origin with $x$. A set $L\subset \rn$ is a {\em star set} if $[0,x]\subset L$ for every $x\in L$. 
For a star set $L\subset \rn$, the {\em radial function} $\rho_L\colon\sn \to\rp$ is given by
$$
\rho_L(s)=  \sup \{\lambda\geq 0 : \lambda s\in L\}.
$$
A star set $L$ is called a {\em star body} whenever $\rho_L \in \cfn$. Conversely, for $f\in \cfn$, there exists a star body $L_f$ such that $f=\rho_{L_f}$.
Let $\starn$ denote the set of star bodies in $\rn$. Note that star bodies are always bounded. 

For $K,L\in \starn$, the {\em radial sum} $K \ra L$ is defined as the star body with radial function satisfying
$$
\rho_{K\ra L}=\rho_K+\rho_L.
$$
The \emph{radial metric} is defined by
$$
\delta(K,L)=\inf\{\varepsilon\geq 0 : K\subset L\ra  \varepsilon B^n, L\subset K\ra  \varepsilon B^n\},
$$
where $B^n$ denotes the Euclidean unit ball of $\rn$. It is easy to check that
$$
\delta(K,L)=\sup\{|\rho_K(s)-\rho_L(s)|: s\in \sn \}=\|\rho_K-\rho_L\|_{\scriptscriptstyle \infty}.
$$
We equip $\starn$ with the topology induced by the radial metric.

\subsection{Valuations on star bodies}
Given an additive abelian semigroup $X$, a function $\oZ\colon\starn\to X$ is a {\em valuation} if 
\begin{equation}\label{eq:valuation}
\oZ(K\cup L)+\oZ(K\cap L)=\oZ(K)+ \oZ(L)
\end{equation}
for all $K,L\in\starn$.

Note that if $K, L \in\starn$, then both $K\cup L$ and $K\cap L$ are in $\starn$. It is easy to see that
$$
\rho_{K\cup L}=\rho_K\vee \rho_L, \hspace{1cm} \rho_{K\cap L}=\rho_K\wedge \rho_L,
$$
where for functions $f, g\colon\sn \to \R$ and $s\in \sn $, we set
$$
(f\vee g)(s)=\max \{f(s), g(s)\},
$$
$$
(f\wedge g)(s)=\min \{f(s), g(s)\}.
$$
Therefore, we can identify  the valuation  $\oZ\colon\starn\rightarrow X$  with the valuation (in the function space sense)  $ \oV\colon \cfn\rightarrow X$,
with the natural identification given by
\begin{equation}\label{eq_star}
 \oV(f)=\oZ(L_f),
\end{equation}
where $L_f$ is the star body with radial function $\rho_{L_f}=f$. In the function space setting, the valuation property \eqref{eq:valuation} becomes
\begin{equation}\label{eq:valuationfunctions}
\oV(f\vee g)+\oV(f\wedge g)=\oV(f)+ \oV(g)
\end{equation}
for every $f,g \in \cfn$.

\subsection{Measure-valued valuations}\label{ss:measure_valued}
General valuations are often hard to describe, so continuity assumptions are used to single out the most relevant classes of valuations. In this paper, we focus on weak$^*$~continuous valuations on star bodies taking values in the space of measures on the sphere.

Let $M(\sn )$ be the space of signed Radon measures on $\sn $, equipped with the variation norm $\|\mu\|=|\mu|(\sn )$. The classical Riesz representation theorem (cf.\ \cite[Theorem 7.2.8]{C}) states that $M(\sn )$ can be identified with the Banach space dual of $C(\sn )$. The duality pairing is given by 
\begin{equation}\label{eq:pairing}
\mu(f)=\langle \mu, f\rangle=\int_{\sn } f(s) \d\mu(s)
\end{equation}
for $f\in C(\sn )$ and $\mu\in M(\sn )$.

\goodbreak
We recall that, given a Banach space $X$, its dual $X^*$ consists of all bounded linear functionals $x^*\colon X\to \R$, and the weak$^*$ topology in $X^*$ is defined as the coarsest topology making continuous the evaluation on points of $X$, in other words, the topology induced by neighborhoods of the form $$\{x^*\in X^*:\max_{1\leq i\leq m} |x^*(x_i)-x_0^*(x_i)|<\varepsilon\}$$ for every $(x_i)_{i=1}^{m}\subset X$, $x_0^*\in X^*$ and $\varepsilon>0$.

We say that a valuation $\oV\colon C(\sn )\rightarrow M(\sn )$ is \emph{weak$^{\,*}$ continuous} if it is continuous when we consider the norm topology in $C(\sn )$ and the weak$^*$ topology in $M(\sn )$, or, equivalently, if the scalar-valued valuation 
$\oV_g\colon C(\sn )\rightarrow \R,$
defined  by 
$$\oV_g(f)=\langle \oV(f),g\rangle=\int_{\sn } g(s)\d \oV(f)(s),$$
is continuous for every $g\in C(\sn )$.

We denote the orthogonal group in dimension $n$ by $\on$. 
Given a function $f\colon \sn \to \R$ and $\phi \in \on$, we will often need the rotation of $f$ by $\phi$, which we will denote by $\phi f= f\circ \phi^{-1}$. 

Similarly, given $A\in \Sigma_n$, the rotation of $A$ by $\phi$ corresponds to the Borel set $\phi A=\{\phi(t)\colon t\in A\}$, which in terms of the corresponding characteristic functions $\chi_A$ can be written as $\phi\chi_{A} = \chi_{\phi A}$. 

We say that a valuation $\oV\colon C(\sn )\to M(\sn )$ is {\em rotation equivariant} if 
$$\oV(\phi f)(\phi B)=\oV(f)(B)$$
for every $f\in C(\sn )$, $B\in \Sigma_n$, and $\phi\in \on$;  
equivalently, by the Riesz representation theorem,
$$\langle \oV(\phi f), \phi g\rangle =\langle \oV(f),g\rangle$$
for every $f,g \in C(\sn )$ and $\phi\in \on$.

It is convenient to work with valuations defined on the cone of non-negative, continuous functions, $\cfn$. We remark that there is more than one way to extend a valuation $\oV\colon \cfn\to M(\sn )$ to a map defined on the whole space $C(\sn )$: define $\oV^\pm\colon C(\sn )\to M(\sn )$ by 
$$\oV^\pm(f)=\oV(f\vee 0)\pm \oV((-f)\vee 0).$$
It is straightforward to check that $\oV^\pm$ also satisfies the valuation property \eqref{eq:valuationfunctions}. Conversely, given a valuation  $\oV\colon C(\sn )\to M(\sn )$, we define the maps $\oV_\pm\colon\cfn\to M(\sn )$  by $$\oV_\pm(f)=\oV(\pm f).$$
It is clear that $\oV_\pm$ are valuations. Assuming that $\oV(0)=0$, we can always write
\begin{equation}\label{e:decomp}
\oV(f)=\oV_+(f_{\scriptscriptstyle +})+\oV_-(f_{\scriptscriptstyle-}),
\end{equation}
where $f_{\scriptscriptstyle +}=f\vee 0$ and $f_{\scriptscriptstyle-}=(-f)\vee0$.

We remark that the classification results that we will derive for valuations on $\cfn$ can be immediately extended to classification results for  valuations on $C(\sn)$ with values in the space of signed Radon measures on $\sn$ such that $\oV(0)=0$ and $\oV(f)\in\measn$ for non-negative $f\in C(\rn)$ using the above decomposition. Therefore, in this work, we will focus on valuations $\oV:C(\sn)_{\scriptscriptstyle +}\to M(\sn)_{\scriptscriptstyle +}$ which satisfy $\oV(0)=0$.

\subsection{Strong Carath\'eodory property}\label{ss:SC} The following notion appeared in \cite{DrewOrlicz}. Let $\mu\in M(\sn)$ be given. A map
$K\colon \rp \times \sn  \to \R$ is \emph{$\mu$-strong Carath\'eodory} if the mapping $s\mapsto K(\lambda, s)$ is Borel measurable on $\sn$ for every $\lambda \in\rp$ and there exists a Borel set $S_0\subset\sn$ with $\mu(S_0)=\mu(\sn)$ such that for every $s\in S_0$ the mapping $\lambda\mapsto K(\lambda,s)$ is continuous on $\rp$. If $K$ is $\mu$-strong Carath\'eodory on $ \rp \times \sn$, then, for every $f\in C(\sn)$, the mapping
\[
    s \mapsto K(f(s), s)
\]
is Borel measurable. 

In the sequel we will need to extend this definition to cover the case of functions $U:\rp\times \sn\to M(\sn)_+$ where in $M(\sn)_+$ we consider the weak$^*$ topology. We will say that $U$ thus defined is $\mu$-strong Carath\'eodory if, for every $g\in C(\sn)$, the map $U_g \colon \rp \times \sn  \to \R$ given by $U_g(\lambda, s) = U(\lambda, s)(g)$ is $\mu$-strong Carath\'eodory.

\section{Vector-Valued Valuations on Star Bodies} \label{s:vvv}

This paper aims to classify measure-valued valuations defined on star bodies. Still, some preliminary results are better understood when stated and proved in slightly greater generality. 

In this section, we will focus on valuations defined on $C(\sn)$ with values on an arbitrary dual Banach space $X^*$ (with the dual norm $\Vert \cdot\Vert_*$), that is, valuations $V:C(\sn)\to X^*$. In the next sections we will  restrict to the case $X^*=M(\sn )$.

We say that a valuation is $\oV\colon C(\sn )\to X^*$ is {\em bounded on bounded sets} if, for every $\lambda>0$, there exists $\gamma\geq 0$ such that
$$\|\oV(f)\|_*\leq \gamma$$
for every $f \in C(\sn )$ with $\|f\|_{\scriptscriptstyle \infty}\leq \lambda$.

By \cite[Lemma 3.1]{TrViJMAA}, every continuous, scalar-valued valuation on $C(\sn )$ is bounded on bounded sets. It follows from this fact and the Uniform Boundedness Principle, that weak$^*$~continuous valuations are bounded on bounded sets:

\begin{proposition}\label{p:bbs}
Let $X^*$ be a dual Banach space. If  $\,\oV\colon C(\sn )\to X^*$ is a weak$^*$~continuous valuation, then $\oV$ is bounded on bounded sets. 
\end{proposition}

\begin{proof} Given $\lambda>0$, define $\mathcal M\subset X^*$ by $\mathcal M=\{\oV(f)\colon\|f\|_{\scriptscriptstyle \infty}\leq \lambda\}$. For every $x\in X$, the scalar-valued valuation $\oV_x\colon C(\sn )\to \R$, given by $\oV_x(f)=\langle \oV(f),x\rangle$, is continuous. Therefore, it follows from  \cite[Lemma 3.1]{TrViJMAA} that 
$$
\sup\{|\langle x^*,x\rangle|:x^*\in \mathcal M\}<\infty.
$$
Since this holds for every $x\in X$, it follows from the Uniform Boundedness Principle (cf.\ \cite[Theorem 3.15]{fabianetal}) that $\mathcal M$ is bounded in norm.
\end{proof}

The main tools to describe and classify continuous, scalar-valued valuations on star bodies in \cite{TrViJMAA, TrViIntRep, Vi} came from Measure Theory. Specifically, the Riesz representation theorem for linear functionals on $C(\sn )$ was used several times. 

A key idea in those papers was to use the fact that the values of a valuation on continuous functions which approximate pointwise a characteristic function of the form $\lambda \chi_A$, with $\lambda\in\R_{\scriptscriptstyle +}$ and $A\in\Sigma_n$, can be used to define a family of measures $\omega_\lambda$ on $\Sigma_n$ which depend continuously on $\lambda$. This is essentially the same as extending a valuation $\oW\colon C(\sn )\to \R$ to a valuation $\overline{\oW}\colon B(\sn )\to \R$ in a canonical way.  

In this paper, we will use the vector-valued version of this approach. Let us start by considering scalar-valued valuations obtained by composing with elements in $X$.

Given a weak$^*$~continuous valuation $\oV\colon C(\sn )\to X^*$ and $x\in X$, let $\overline{\oV}_x\colon B(\sn )\to \R$ be the canonical extension of $\oV_x=\langle \oV,x\rangle$ to the space of bounded Borel functions as constructed in \cite{TrViJMAA}. For completeness, we outline the extension construction here without providing proofs of the statements. Details can be found in \cite{TrViJMAA}.

First, by \cite[Theorem 1.3]{TrViJMAA}, every valuation $\oW\colon C(\sn )\to \R$ can be written as a difference of non-negative valuations, that is, there exist non-negative valuations $\oW^+$ and $\oW^-$ such that  $\oW=\oW^+- \oW^-$ (however, this decomposition need not be unique). The following construction defines a Radon measure for every $\lambda > 0$ and each sign $\pm$: For $A\in\Sigma_n$, we set
\begin{equation}\label{eq:def nu lambda+}
\omega_{\lambda}^{\pm}(A)=
\inf_{\substack{\\ G \text{ open} \\ A\subset G}}\,\,
\sup_{\substack{K \text{ closed}\\[1pt]
K\subset G}} \inf\Big\{\oW^{\pm}(f)\colon f\in C(\sn ),\,\chi_K\leq \frac{f}{\lambda} \leq \chi_G \Big\},
\end{equation}
and 
\begin{equation}\label{eq:def nu lambda}
    \omega_{\lambda}=\omega^+_{\lambda}-\omega^-_{\lambda}.
\end{equation}

For a simple function $g=\sum_{i=1}^m \lambda_i\chi_{A_i}$ with pairwise disjoint $A_i\in \Sigma_n$ and $\lambda_i >0$, we define 
$$\overline{\oW}(g)=\sum_{i=1}^m\omega_{\lambda_i}(A_i).$$ 
This definition admits a unique continuous extension to $B(\sn )$, and it can be checked that 
$\overline{\oW}(f)=\oW(f)$ for every $f\in C(\sn )$ (see \cite[Theorem~6.1]{TrViJMAA}). Note that the extended valuation is continuous (with respect to the norm $\|\cdot\|_{\scriptscriptstyle \infty}$). In particular, the following identity will be useful (see \cite[Lemma~5.6]{TrViJMAA}): for every closed set $C\subset \sn $ and $\lambda \ge 0$, we have
\begin{equation}\label{limitrims}
\omega_{\lambda}(C)=\lim_{i\to \infty}  \oW(f_i)
\end{equation}
whenever $f_i\in \cfn$ are such that $\|f_i\|_{\scriptscriptstyle \infty}\leq \lambda$,  $f_i(s)=\lambda$ for $s\in C$, and $f_i(s)=0$ when $\dist(s,C)\geq1/i$ (where $\dist(s,C)$ is the infimum of the spherical distances of $s$ and $t\in C$).

\goodbreak
Next, we show that the extension procedure is linear. 

\goodbreak
\begin{lemma}\label{l:ABextension}
For valuations $\oW\colon C(\sn )\to \R$, the extension procedure from $\oW$ to $\overline{\oW}\colon B(\sn )\to \R$ is linear in $\oW$, that is, 
$$\overline{\alpha \oW_1+\beta \oW_2}=\alpha \overline{\oW_1}+ \beta \overline{\oW_2}$$
for every $\alpha, \beta \in \R$ and scalar-valued valuations $\oW_1, \oW_2$ on $C(\sn)$.
\end{lemma}
\begin{proof}
For every $\lambda \in \R_{\scriptscriptstyle +}$, let $\omega^1_\lambda$ and $\omega^2_\lambda$ be the measures induced by $\oW_1$ and $\oW_2$, respectively, in the procedure described above, and let $\omega_\lambda$ be the measure induced by $\alpha \oW_1+\beta \oW_2$. For every closed set $C$ and every sequence $(f_i)_{i\in\mathbb N}$ in $\cfn$ such that  $\|f_i\|_{\scriptscriptstyle \infty}\leq \lambda$, $f_i(s)=\lambda$ for $s\in C$ and $f_i(s)=0$ for $\dist(s,C)\geq1/i$, we obtain by \eqref{limitrims} that
\begin{align*}
\omega_\lambda(C)&=\lim_{i\rightarrow \infty} (\alpha \oW_1+\beta \oW_2)(f_i)\\ &=\alpha \lim_{i\rightarrow \infty} \oW_1(f_i)+\beta \lim_{i\rightarrow \infty} \oW_2(f_i)=\alpha\, \omega^1_\lambda(C)+\beta\, \omega^2_\lambda(C).
\end{align*}
By the regularity of the measures involved, the same identity holds for arbitrary sets in $\Sigma_n$. Hence, the linearity of the extension procedure holds on simple functions. The result now follows from the continuity of the extended valuations and density of simple functions in $B(\sn)$.
\end{proof}

In the following, we introduce  scalar measures associated to a weak$^*$~continuous valuation $\oV\colon C(\sn )\to X^*$ with $\oV(0)=0$. For $\lambda\in \R_{\scriptscriptstyle +}$ and $x\in X$, let 
\begin{equation}\label{eq:defnulambda}
    \nu_{\lambda,x}= \omega_\lambda
\end{equation} 
be the scalar-valued measure associated to $\oW=\langle\oV,x\rangle$ in \eqref{eq:def nu lambda}. 

Now, given $A\in\Sigma_n$, we consider the linear and continuous mapping $x\mapsto \nu_{\lambda,x}(A)$. The linearity of this mapping follows from Lemma \ref{l:ABextension} and the continuity from the definition of $\nu_{\lambda,x}(A)$ and Proposition \ref{p:bbs}. Note also, that for a fixed $x\in X$ the mapping $A\mapsto \nu_{\lambda,x}(A)$ is a measure on $\Sigma_n$. 
\goodbreak
We will use the following notion of control measure.  
\begin{definition}\label{d:Mcontrol}
Let $\oV:C(\sn)\to X^*$ be a weak$^*$ continuous valuation. For every $x\in X$ and $\lambda\geq0$, let $\nu_{\lambda,x}:\Sigma_n\to\mathbb{R}$ denote the associated (and previously defined) measures. A scalar measure $\mu:\Sigma_n\to\mathbb{R}$ is said to be a control measure of $\oV$ if there is $\zeta:\mathbb{R}_{\scriptscriptstyle +}\to\mathbb{R}_{\scriptscriptstyle +}$, called control function, such that
\begin{equation}\label{i:controlmeasure}
        |\nu_{\lambda,x}(A)|\leq \|x\| \zeta(\lambda)\mu(A)
\end{equation}
for every $\lambda\geq0$, $A\in \Sigma_n$ and $x\in X$, where $\|x\|$ denotes the norm of $x$ in $X$.
\end{definition}

\section{Construction of a control measure and further results from the scalar case}\label{S:control}
Throughout the remainder of the paper we restrict to the case where $X=C(\sn)$ and denote by $\|\cdot \|$ the dual norm on $M(\sn)$, dropping the subscript $*$. A key difference with respect to the scalar case is that while every scalar valuation can be written as the difference of two non-negative valuations, this property fails for vector-valued valuations. To see this, it suffices to consider a valuation defined by a continuous linear operator $T:C(S^{n-1})\to M(S^{n-1})$  that is not absolutely summing (hence, it cannot be written as a difference of positive linear operators, see \cite[Chapter 6]{Diestel-Faires}). Thus, from now on we focus on 
valuations $\oV:C(\sn)_{\scriptscriptstyle +}\to M(\sn)_{\scriptscriptstyle +}$ that are weak$^*$ continuous and satisfy $\oV(0)=0$. 

\goodbreak
In previous works \cite{TrViJMAA,TrViIntRep}, a control measure was constructed for scalar valuations. Let us first recall this construction. Given a non-negative 
valuation $\oW:C(\sn)\to\mathbb{R}_{\scriptscriptstyle +}$ and $\lambda_0\geq0$, a Radon measure $\mu_{\lambda_0}:\Sigma_n\to\mathbb{R}$ can be defined such that on open sets $G\subset \sn$ 
\begin{equation}\label{eq:control measure}
\mu_{\lambda_0}(G)=\sup\{\oW(f)\colon  f\prec G, \, \|f\|_{\scriptscriptstyle \infty}\leq \lambda_0\}, 
\end{equation}
where $f\prec G$ means that the support of $f$ is contained in $G$.  

The key property of $\mu_{\lambda_0}$, and the one needed for the arguments in the previous works \cite{TrViJMAA, TrViIntRep}, is its ``control'' property: for every $A\in\Sigma_n$ and $f\in C(\sn)$,
\begin{equation}\label{eq:control}
 f\prec A \,\text{ and }\,\|f\|_{\scriptscriptstyle \infty}\leq\lambda_0\,\text{ imply }\,|\oW(f)|\leq \mu_{\lambda_0}(A). 
\end{equation}
Suppose that $f\in \cfn$. Let us consider the constant one function $\uno$. It follows from the positivity of $\oV$ that
$$
0\leq \langle \oV(f),g\rangle\leq \langle \oV(f),\|g\|_{\scriptscriptstyle\infty}\uno\rangle
$$ 
for $g\ge0$.
Hence, 
\begin{equation}\label{Eq:inequality}
|\oV_g(f)|=|\langle \oV(f),g\rangle|=|\langle \oV(f),g_+-g_-\rangle|\leq \langle \oV(f),|g|\rangle\leq \oV_{\uno}(f)\,\|g\|_{\scriptscriptstyle \infty}
\end{equation}
for $g\in C(\sn )$.

Let $\mu_{\lambda_0,\uno }$ be the measure defined in \eqref{eq:control measure} associated to the scalar-valued valuation $\oV_{\uno}=\langle \oV, \uno\rangle$. It follows from \eqref{Eq:inequality} that for every $g\in C(\sn)$, $f\in\cfn$ and $A\in\Sigma_n$ with $\|f\|\leq \lambda_0$ and $f\prec A$, we obtain
\[
    |V_g(f)|\leq V_{\uno} (f)\|g\|_{\scriptscriptstyle \infty} \leq \mu_{\lambda_0,\uno} (A)\|g\|_{\scriptscriptstyle \infty}.
\]
In order to handle arbitrarily large functions $f\in C(\sn )$, we will use the following measure:
\begin{equation}\label{eq_deK_mu}
\mu=\sum_{i=1}^\infty\frac{\mu_{i,\uno}}{2^i\|\mu_{i,\uno}\|},    
\end{equation}
where $\|\mu_{i,\uno}\|$ denotes the variation norm. Note that the above series is absolutely convergent and, therefore, defines a proper, non-negative, regular measure. Although the measure $\mu$ may be defined on a larger $\sigma$-algebra, we restrict to Borel sets, since this is sufficient for our purposes. 

Now let us see that the measure $\mu$ is a control measure in the sense of \eqref{i:controlmeasure}. It follows easily that if $f\in C(\sn)_{\scriptscriptstyle +}$, $g\in C(\sn)$ and $A\in\Sigma_n$ with $f\prec A$ one has
\begin{equation}\label{Eq:relating}
    |V_g(f)|\leq \|g\|_{\scriptscriptstyle \infty} \mu_{m,\uno}(A) \leq \|g\|_{\scriptscriptstyle \infty} 2^m \|\mu_{m,\uno}\|\mu(A)
\end{equation}
for every $m\in\mathbb{N}$ such that $\|f\|_{\scriptscriptstyle \infty}\leq m$. Define now the function $\zeta$ as follows 
\begin{equation}\label{eq_cm}
\cm(t)=2^{\lceil t\rceil}\|\mu_{\uno, \lceil t\rceil}\|.
\end{equation}
where $\lceil t\rceil=\min\{m\in\mathbb{Z}:t\leq m\}$ denotes the ceiling function. 

Finally, to check that $\mu$ is a control measure we have to verify inequality \eqref{i:controlmeasure}. Let $C\subset \sn$ be a closed set, let $\lambda> 0$ and let $g\in C(\sn)$. Consider a sequence $\{f_m\}_{m\in\mathbb N}\subset\cfn$ with  $\|f_m\|_{\scriptscriptstyle \infty}\leq \lambda$, $f_m(s)=\lambda$ for $s\in C$ and $f_m(s)=0$ whenever $\dist(s,C)\geq1/m$. By \eqref{limitrims},
\[
    |\nu_{\lambda,g}(C)|=|\overline\oV_g(\lambda \chi_C)|=\lim_{m\to\infty}|\oV_g(f_m)|\leq \|g\|_{\scriptscriptstyle \infty} \zeta(\lambda)\mu(C).
\]
The regularity of the measures implies now that the same inequality holds for Borel sets. That is, for every $A\in\Sigma_n$, 
\[
    |\nu_{\lambda,g}(A)|\leq \|g\|_{\scriptscriptstyle \infty} \zeta(\lambda)\mu(A).
\]
Hence, $\mu$ is a control measure with $\zeta$ as the control function. 
\goodbreak

We state for further reference the following trivial fact about the control function:

\begin{observation}\label{o:zeta}
$\zeta$ is non-decreasing and constant on each interval $[m-1,m)$, $m\in\mathbb{N}$.
\end{observation}

Our next objective is to adapt the integral representation result for scalar-valued valuations from \cite[Theorem 1.1]{TrViIntRep}. In that paper, the representation was obtained for a specific control measure which does not necessarily coincide with our $\mu$. Thus, our goal is to obtain an integral representation of $\oV_g$ with respect to the measure $\mu$ defined in \eqref{eq_deK_mu}. This representation with $\mu$ follows directly from the proof of \cite[Proposition 2.2]{TrViIntRep} where it becomes apparent that the only requirement on the measure $\mu$ is the control property. In particular, one can prove the following result.

\begin{lemma}\label{l:weak pseudokernel}
Let $\oV\colon C(S^{n-1})_+\rightarrow M(S^{n-1})_{\scriptscriptstyle +}$ be a weak$^*$ continuous valuation, and consider $\mu$ as above. For $g\in C(\sn)$, there is a $\mu$-strong Carath\'eodory mapping $K_g\colon \rp\times \sn \to \R$ such that
$$\oV_g(f)= \int_{\sn } K_g(f(s), s) \d\mu(s)$$
for every $f\in \bfn$. Moreover, $K_g$ can be chosen to satisfy the following conditions.
\begin{itemize}
    \item For each $\lambda\ge 0$, the function $K_g(\lambda,\cdot)$ coincides with the Radon--Nikodym derivative with respect to $\mu$ of the measure $\nu_{\lambda,g}$ associated to the scalar-valued valuation $\oV_g$ as in \eqref{eq:defnulambda}. In particular, we have $K_g(\lambda, \cdot)\in L_1(\sn,\mu)$.
    \item The function $K_g(\cdot,s)$ is continuous for every $s\in \sn $.
\end{itemize}
\end{lemma} 
\begin{proof}
For the sake of completeness we include the details in Appendix \ref{A:Inspection}.
\end{proof}

We will repeatedly rely on the fact that $K_g(\lambda, \cdot)$ belongs not only to $L_1(\sn,\mu)$, but also to $L_{\scriptscriptstyle \infty}(\sn,\mu)$. This will be established in the next lemma. 

\begin{lemma}\label{l:bounded kernel}
If  $g\in C(\sn)$, then the function $K_g\colon \rp\times \sn \to \R$ in Lemma \ref{l:weak pseudokernel} satisfies 
$$K_{g}(\lambda, s) \leq \cm(\lambda) \|g\|_{\scriptscriptstyle \infty}$$
for $\mu$-almost every $s\in \sn $ and for each $\lambda\in\rp$. 
\end{lemma}
\begin{proof}
We consider $K_g(\lambda, \cdot)$ as an element of $L_1(\sn,\mu)$. On the one hand,  it follows from the definition of Radon--Nikodym derivative and Lemma \ref{l:weak pseudokernel} that  
$$ \nu_{\lambda,g}(A) = \int_A K_g(\lambda, s) \d\mu(s)$$
for every Borel set $A\subset \sn $.
On the other hand, for every $\varepsilon>0$, it follows from the regularity of $\mu$ that there is an open set $G\supset A$  with 
\begin{equation}\label{eq:mureg}
 {\|g\|_{\scriptscriptstyle \infty} \cm(\lambda)}\, \mu(G\setminus A)\leq {\varepsilon}.   
\end{equation} 
It also follows from the definition of $\nu_{\lambda,g}$ in \eqref{eq:def nu lambda+} that there is a closed set  $C\subset G$ such that 
\begin{equation}
    \nu_{\lambda,g}(A) \leq \inf \{\langle \oV(f),g\rangle: f\in C(\sn ),\,\lambda\chi_C\leq f\leq \lambda\chi_G\}+\varepsilon.
\end{equation} 
Since we have $\mu_{\uno,\lambda}(G)=\sup\{\langle \oV(f),\uno\rangle:f\in C(\sn ),\,0\leq f\leq \lambda\chi_G\}$ and $\langle \oV(f),g\rangle\leq\|g\|_{\scriptscriptstyle \infty} \langle \oV(f),\uno\rangle$, it follows that
\begin{equation}\label{eq:nug2}
    \nu_{\lambda,g}(A) \leq \|g\|_{\scriptscriptstyle \infty} \cm(\lambda)\mu(G)+\varepsilon.
\end{equation} 
Combining \eqref{eq:mureg} and \eqref{eq:nug2}, we have 
$$\nu_{\lambda,g}(A)  \leq \|g\|_{\scriptscriptstyle \infty} \cm(\lambda)\mu(G) + \varepsilon \leq  \|g\|_{\scriptscriptstyle \infty} \cm(\lambda)\mu(A) + 2\varepsilon.$$
Since this happens for every $\varepsilon>0$, we obtain that  
$$\int_A K_g(\lambda, s) \d\mu(s)= \nu_{\lambda,g}(A) \leq \|g\|_{\scriptscriptstyle \infty} \cm(\lambda)\, \mu(A) $$
for every Borel set $A\subset \sn $, which proves the result.
\end{proof}

\begin{lemma}\label{l:kernel difference}
For every $g, h\in C(\sn )$ and $\alpha\in\mathbb{R}$, the function $K_g\colon \rp\times \sn \to \R$ from Lemma \ref{l:weak pseudokernel} satisfies 
$$K_{g}(\lambda, s) +K_{h}(\lambda, s) = K_{g+h}(\lambda, s)\quad \text{and}\quad K_{\alpha g}(\lambda,s)=\alpha K_g(\lambda,s)$$
for $\mu$-almost every $s\in \sn $ and for each $\lambda\in\rp$. 
\end{lemma}

\begin{proof}
It follows from the Radon--Nikodym representation theorem that
\begin{align*}
    \int_A K_{g+h}(\lambda, s) \d\mu(s) &= \nu_{\lambda,g+h}(A) \\
    &=\nu_{\lambda,g}(A)+\nu_{\lambda,h}(A)\\ 
    &= \int_A K_g(\lambda, s) \d\mu(s) + \int_A K_h(\lambda, s) \d\mu(s) \\
    &= \int_A ( K_g(\lambda, s)+ K_h(\lambda, s)) \d\mu(s)
\end{align*}
for every Borel set $A\subset\sn$. Consequently, the first equality holds $\mu$-almost everywhere in $s\in\sn$. The second identity follows analogously.
\end{proof}

\section{Characterization of weak$^*$ continuous valuations}\label{S:Usection} 

Let $\poly$ be the set (of restrictions to $\sn$) of polynomials with rational coefficients defined on $\rn$. It follows from the Stone--Weierstrass Theorem that $\poly$ is a (countable) dense subset of $C(\sn )$. For our purposes, we could actually take any countable dense set in $C(\sn )$, closed under sums and multiplication by rational numbers, instead of $\poly$.   
Let $\left( p_j\right)_{j\in \mathbb N}$ be an enumeration of $\poly$ and let $(\lambda_m)_{m\in\mathbb{N}}$ be an enumeration of $\mathbb{Q}_{\scriptscriptstyle +}$. We consider now the functions $K_{p_j}$ as in Lemma \ref{l:weak pseudokernel}.

For every pair $(p_i,p_j)\subset \poly$ and every $m\in \mathbb N$, let us consider 
\[
  S^{m}_{i,j}= \{s\in\sn:K_{p_i}(\lambda_m,s)+K_{p_j}(\lambda_m,s)\neq K_{p_i+p_j}(\lambda_m,s)\}.
\]
By Lemma \ref{l:kernel difference}, we obtain that $S^{m}_{i,j}$ is Borel, $\mu(S^m_{i,j})=0$ and, therefore, the set $S_1=\bigcup_{i,j,m\in\mathbb{N}}S^m_{i,j}$ is also Borel with $\mu(S_1)=0$. 

Observe that for each $s\notin S_1$ and all $i,j,m\in\mathbb{N}$, the relation $K_{p_i}(\lambda_m,s)+K_{p_j}(\lambda_m,s)=K_{p_i+p_j}(\lambda_m,s)$ holds.

Now, for $\lambda\not\in\mathbb{Q}_{\scriptscriptstyle +}$, consider $\{\lambda_m\}_{m\in\mathbb{N}}\subset\mathbb{Q}_{\scriptscriptstyle +}$ converging to $\lambda$.
Since $K_{p}(\cdot,s)$ is continuous for every $s\not\in S_1$ and every $p\in\poly$,  for every $s\not\in S_1$,
\begin{align*}
    K_{p_i}(\lambda,s)+K_{p_j}(\lambda,s)&=\lim_{m\to\infty}K_{p_i}(\lambda_m,s)+K_{p_j}(\lambda_m,s)\\
    &=\lim_{m\to\infty}K_{p_i+p_j}(\lambda_m,s)=K_{p_i+p_j}(\lambda,s).
\end{align*}
Hence, $p_i\mapsto K_{p_i}(\lambda,s)$ is additive for every $\lambda\geq 0$ and every $s\not\in S_1$. It is well known that every additive function is $\mathbb{Q}$-linear. Therefore, 
\[
    K_{\alpha p}(\lambda,s)=\alpha K_p(\lambda,s)
\]
for every $\alpha\in\mathbb{Q}$, every $p\in\poly$, every $\lambda\in\mathbb{R}_{\scriptscriptstyle +}$, and every $s\not\in S_1$.

Next, we will make a similar procedure for the boundedness of the mapping. For every $p_i\in\poly$ and every $\lambda_m\in\mathbb{Q}_{\scriptscriptstyle +}$, let 
\[
  S^{m}_i= \{s\in\sn: K_{p_i}(\lambda_m,s) >\zeta(\lambda_m)\|p_i\|_{\scriptscriptstyle \infty} \}.
\]
By Lemma \ref{l:bounded kernel}, $S^{m}_i$ is Borel and $\mu(S^m_i)=0$. Let $S_2=\bigcup_{i,m\in\mathbb{N}}S^m_i$ which is also Borel and satisfies $\mu(S_2)=0$. Again, for $\lambda\in\mathbb{R}_{\scriptscriptstyle +}$, take a sequence $\{\lambda_m\}_{m\in\mathbb{N}}\subset\mathbb{Q}_{\scriptscriptstyle +}$ with $\lambda_m\to\lambda$. Since $K_{p_i}(\cdot,s)$ is continuous, for every $s\not\in S_2$,
\[
    K_{p_i}(\lambda,s)=\lim_{m\to\infty}K_{p_i}(\lambda_m,s)\leq \limsup_{m\to\infty}\zeta(\lambda_m) \|p_i\|_{\scriptscriptstyle \infty}.
\]
By Fact \ref{o:zeta}, we have $\limsup_{m\to\infty}\zeta(\lambda_m)\leq \zeta(\lambda)$. Thus, for every $s\not\in S_2$ and every $p_i\in\poly$, we have
\[
    K_{p_i}(\lambda,s)\leq \zeta (\lambda)\|p_i\|_{\scriptscriptstyle \infty}.
\]
Finally we define 
\begin{equation}\label{eq_set}
  S=\sn\setminus \left(S_1\cup S_2\right),
\end{equation}
which is Borel and has $\mu$-full measure.

For every $\lambda\geq 0$ and every $s\in S$, consider the mapping 
\[
    \begin{array}{ccl}
         \poly &\longrightarrow &\mathbb{R} \\
         p_i&\longmapsto &K_{p_i}(\lambda,s),
    \end{array}
\]
which is a linear and bounded with respect to the uniform norm, by the previous discussion. Therefore, we can consider its linear and continuous extension 
\begin{equation}\label{def:U}
U(\lambda,s):C(\sn)\to\mathbb{R}    
\end{equation} 
which satisfies the following fact.

\begin{observation}\label{ob:bounded}
For every $s\in S$, we have $\|U(\lambda, s)\|\leq \cm(\lambda)$.
\end{observation}

For every $s\in S$, we obtain a continuous linear functional such that
\[
    U(\lambda,s)(p_i)=\langle U(\lambda,s),p_i\rangle=K_{p_i}(\lambda,s)
\]
for every $p_i\in\poly$ and every $\lambda\geq0$. Hence, for every $A\in\Sigma_n$
\[
    \nu_{\lambda,p_i}(A)=\int_AK_{p_i}(\lambda,s)\d \mu(s)=\int_A U(\lambda,s)(p_i)\d \mu(s).
\]

Next, we show that a similar result holds not only for $p_i\in\mathcal P_{\mathbb Q}$ but also for every $g\in C(\sn )$.
Specifically, we show that $s\mapsto U(\lambda,s)(g)$, considered as an element in $L_1(\sn,\mu)$, coincides with the Radon--Nikodym derivative of $\nu_{\lambda,g}$, that is, $U(\lambda, s)(g)=K_{g}(\lambda,s)$ as elements of $L_1(\sn,\mu)$. 

\begin{lemma}\label{l:Uforg}
Let $g\in C(\sn )$. For every $A\in \Sigma_n$, we have 
$$
\nu_{\lambda,g}(A) = \int_A K_{g}(\lambda, s) \d \mu(s) = \int_A U(\lambda,s)(g) \d \mu(s).
$$
\end{lemma}

\begin{proof}
Let $g\in  C(\sn )$, $\lambda\in\rp$, and let $S\subset \sn $ be defined as in \eqref{eq_set}, and consider the function 
\begin{align*}   
S & \longrightarrow  \R  \\
s &  \longmapsto  U(\lambda, s)(g).    
\end{align*} 
Note that this is well-defined as $U(\lambda,s)\in M(\sn )=C(\sn )^*$ for $s\in S$.

Choose a sequence $(p_i)_{i\in\mathbb N}$ in $\poly$ which converges to $g$ in the supremum norm. Note that, for every $i\in \mathbb N$ and for every $s\in S$, we have $U(\lambda, s)(p_i) = K_{p_i}(\lambda, s)$, and, therefore, $s\mapsto U(\lambda, s)(p_i)$ is Borel. 
Also, note that it follows from the continuity of $U(\lambda,s)$ that $\lim_{i\to\infty} U(\lambda,s)(p_i) = U(\lambda,s)(g)$ for every $s \in S$. 
In particular, the function  $ s  \mapsto  U(\lambda,s)(g)$ is Borel on $\sn$ (where we set $U(\lambda, s)(g)=0$ for $s\notin S$). 

\goodbreak
Fact \ref{ob:bounded} yields that $\|U(\lambda, s)\|\leq \cm(\lambda)$ for every $s\in S$. Therefore, for every $s\in S$ we have
\[\vert U(\lambda,s)(g)\vert \leq \cm(\lambda) \|g\|_{\scriptscriptstyle \infty}.\] 
In particular, the map $ s  \mapsto U(\lambda, s)(g)$ defines an element in $L_{\scriptscriptstyle \infty}(\sn,\mu)\subset L_1(\sn,\mu)$. 

We can choose the sequence $(p_i)_{i\in\mathbb N}$ so that $\|p_i\|_{\scriptscriptstyle \infty}\leq 1+\|g\|_{\scriptscriptstyle \infty}$ for every $i\in \mathbb N$. Hence, the constant function $\cm(\lambda)(1+\|g\|_{\scriptscriptstyle \infty}) \uno$ dominates $ U(\lambda,\cdot)(g)$ and $U(\lambda,\cdot)(p_i)$ for every $i\in \mathbb N$. Finally, by the Dominated Convergence Theorem, for every $A\in\Sigma_n$
\begin{align*}
    \nu_{\lambda,g}(A) & = \lim_{i\to\infty} \nu_{\lambda,p_i}(A)\\
    &= \lim_{i\to\infty} \int_{A}K_{p_i}(\lambda,s) \d\mu(s)\\
    &=\lim_{i\to\infty} \int_{ A} U(\lambda,s)(p_i) \d\mu(s)\\ 
        & = \int_{ A} \lim_{i\to\infty} U(\lambda,s)(p_i) \d\mu(s) =\int_{ A}  U(\lambda,s)(g)\d\mu(s).
\end{align*}
\end{proof}

\begin{lemma}\label{l:Uwcont}
For every $s\in S$ and every $g\in C(\sn)$, the mapping
\[
    \lambda\mapsto U(\lambda,s)(g)
\]
is continuous on $\mathbb{R}_{\scriptscriptstyle +}$.
\end{lemma}

\begin{proof}
Fix $g\in C(\sn )$,  $s\in S$, and $\lambda\ge 0$. Let $(\lambda_i)_{i\in \mathbb N} \subset \rp$ be a sequence which converges to $\lambda$. Without loss of generality, we can assume $\lambda_i <\lambda+1$.

Let $\varepsilon > 0$ and choose $p_j\in\poly$ with $\|g-p_j\|_{\scriptscriptstyle \infty}\leq \varepsilon$. Since $\lambda\mapsto U(\lambda, s)(p_j)$ is continuous, there exists $i_0\in \mathbb N$ such that  
$$\left|  U(\lambda_i, s)(p_j) -   U(\lambda, s)(p_j) \right|\leq \varepsilon$$
for every $i\geq i_0$. Hence, we have
\begin{align*}
   \vert  &U(\lambda_i, s)(g) -  U(\lambda, s)(g)\vert \\
&\leq  | U(\lambda_i, s)(g-p_j) | +   |U(\lambda, s)(g-p_j)| + |U(\lambda_i, s)(p_j) - U(\lambda, s)(p_j)  | \\
 &\leq \|g-p_j\|_{\scriptscriptstyle \infty}(\|U(\lambda_i, s)\|+\| U(\lambda, s)\| ) +|  U(\lambda_i, s)(p_j) - U(\lambda, s)(p_j) | \\
 &\leq  (\cm(\lambda+1)+\cm(\lambda)+1)\varepsilon
\end{align*}
for every $i\geq i_0$, where we used Fact \ref{ob:bounded} and the fact that $\zeta$ is non-decreasing. Since $\lambda$ is fixed and $\varepsilon$ arbitrary, this finishes the proof.
\end{proof}

The previous result shows that for every $g\in C(S^{n-1})$, the function of two variables $$(\lambda, s)\mapsto \langle U(\lambda, s),g\rangle$$ is $\mu$-strong Carath\'eodory, and in particular, it is measurable. Moreover, 
$$|\langle U(\lambda, s),g\rangle|\leq  \|g\|_{\scriptscriptstyle \infty}\, \cm(\lambda)$$ 
for every $\lambda\ge 0$ and $s\in S$. Therefore, the integral $$\int_{\sn } \langle U(f(s),s),g \rangle \d\mu(s)$$ is well-defined for every $f\in \cfn$.
  
We now have all the tools to prove our main result Theorem \ref{th:rep}. First, we state it in the context of  valuations on functions.

\begin{theorem}\label{th:representation}
A map $\oV\colon  \cfn\to \measn$ is a weak$^*$\! continuous valuation with $\oV(0)=0$ if and only if there exist $\mu\in \measn$ and a $\mu$-\somename map $U\colon \rp \times \sn  \to \measn$ with $U(0,s)=0$ for $\mu$-almost every $s\in \sn$, such that
\begin{equation}\label{eq_equality}
\langle \oV(f),g\rangle = \int_{\sn } \langle U(f(s), s),g\rangle \d\mu(s)
\end{equation}
for every $f\in\cfn$ and every $g\in C(\sn)$. 
\end{theorem}

\begin{proof}
Let $\oV\colon \cfn\to \measn$ be a weak$^*$~continuous valuation and let $\mu$ and $\cm$ be defined as in \eqref{eq_deK_mu} and \eqref{eq_cm}, respectively. By Lemmas \ref{l:Uforg} and \ref{l:Uwcont}, the map $U\colon \rp \times \sn  \to \measn$ given by \eqref{def:U} satisfies the $\mu$-strong Carath\'eodory condition with respect to the weak$^*$ topology on $\measn$ (for convenience, set $U(\lambda,s)$ to be the zero measure for $s\notin S$). By Fact \ref{ob:bounded}, $U$ is uniformly controlled by $\cm$.

It remains to prove that \eqref{eq_equality} holds for $f\in\cfn$ and $g\in C(\sn)$. To see this, consider the map $$\oV'\colon\cfn\to \measn$$ defined by 
\eqref{eq_equality}, that is, $$\langle \oV'(f),g\rangle = \int_{\sn } \langle U(f(s), s),g\rangle \d\mu(s)$$
for $f\in\cfn$, $g\in C(\sn)$. First, note that the above integral is well-defined because $U$ satisfies the $\mu$-strong Carath\'eodory condition. Also, the expression $\langle \oV'(f),g\rangle$ is clearly linear in $g$ and bounded because $U$ is uniformly controlled. Hence, $\oV'(f)\in M(\sn )$ for every $f\in\cfn$. 

Note that, $$\langle U(\lambda, s),p\rangle = K_{p}(\lambda, s)$$ for every $p\in \poly$ and  $\mu$-almost every $s\in \sn $. It follows that 
$$\langle \oV(f),p\rangle= \int_{\sn } \!K_p(f(s), s) \d\mu(s) = \int_{\sn }\! \langle U(f(s), s),p\rangle \d\mu(s) = \langle \oV'(f),p\rangle$$
for every $f\in \cfn$ and $p\in \poly$.
Therefore, $\oV'(f)$ and $\oV(f)$ are two elements in $M(\sn )$ which act identically on a dense set of $C(\sn )$. Hence, they coincide.  

For the converse implication, note first that since $U$ is $\mu$-\somename, the integral $\int_{\sn } \langle U(f(s), s),g\rangle \d\mu(s)$ exists for every $g\in C(\sn)$.  Hence, the formula $$
\langle \oV(f),g\rangle = \int_{\sn } \langle U(f(s), s),g\rangle \d\mu(s)$$
defines a map $\oV\colon \cfn\to \measn$. The fact that $\oV(f)$ is an element of $M(\sn )$ follows as above: it defines a linear and bounded function of $g$. 

It is easy to see that $\oV$ is a valuation. Only weak$^*$ continuity remains to be proved. 
Note that the $\mu$-strong Carath\'eodory condition implies that given $g$, there is a set $S\subset \sn $  of full $\mu$-measure such that $\langle U(\lambda, s),g\rangle$ is continuous in $\lambda$ for every $s\in S$.  

Fix $f\in \cfn$ and $g\in C(\sn)$. Let $(f_i)_{i\in \mathbb N}$ be a sequence in $\cfn$ converging to $f$. For every $s\in S$, we have $\langle U(f_i(s), s),g\rangle \rightarrow \langle U(f(s), s),g\rangle$ as $i\to\infty$. Moreover, since $U$ is uniformly controlled, there is a constant such that $\langle U(f_i(s), s),g\rangle \leq C$ and $\langle U(f(s), s),g\rangle \leq C$ for every $i\in \mathbb N$ and $s\in S$. 
Now, the Dominated Convergence Theorem yields that 
\begin{align*}
    \lim_{i\to\infty} \langle \oV(f_i),g\rangle &= \lim_{i\to\infty} \int_{\sn } \langle U(f_i(s), s),g\rangle \d\mu(s) \\&= \int_{\sn } \langle U(f(s), s),g\rangle \d\mu(s) \\[4pt]
    &=\langle \oV(f),g\rangle.
\end{align*} 
As a final step, note that $\left\langle\oV(0),g  \right\rangle=0$ for every $g\in C(\sn)$.
\end{proof}

Clearly, from this result, Theorem \ref{th:rep} follows easily. In fact, the integral representation from Theorem \ref{th:representation} provides extra information about the original valuation. Recall that $B(\sn )$, the space of bounded Borel functions, can be identified with a subspace of $M(\sn )^*=C(\sn )^{**}$, with duality given by $\langle g,\mu\rangle=\int_{\sn} g \d\mu$ for $\mu\in M(\sn )$ and $g\in B(\sn )$:

\begin{corollary}\label{co: extension}
If $\,\oV\colon \cfn\to \measn$ is a weak$^*$~continuous valuation with $\oV(0)=0$, then the integral representation in Theorem \ref{th:representation} defines a weak$^*$~continuous extension of $\oV$ to a valuation $\overline{\oV}\colon \bfn\to \measn$.
Moreover, this extension satisfies that 
$$\lim_{i\to\infty} \langle \overline{\oV}(h_i),g\rangle= \langle \overline{\oV}(h),g\rangle$$
for every bounded sequence $(h_i)_{i\in \mathbb N}$ in $\bfn$ which is pointwise convergent to $h\in \bfn$ and for every $g\in C(\sn )$.
\end{corollary}

\begin{proof}
First, recall that, for every $g\in C(\sn )$,  the function 
$$\rp\times \sn \to \R$$
$$(\lambda, s) \mapsto \langle U(\lambda, s),g\rangle$$ is $\mu$-strong Carath\'eodory. Hence, for every $h\in B(\sn )$, the function $s\mapsto \langle U(h(s), s),g\rangle$ is measurable as a composition of measurable functions. 
The rest of the statements follow immediately, using the Dominated Convergence Theorem. 
\end{proof}

In particular, this means that measure-valued weak$^*$ continuous valuations on star bodies can be extended to bounded Borel star sets preserving continuity. The following is an elementary observation that we include for the sake of clarity.

\begin{corollary}\label{co:uniqueness} 
Let $\mu\in M(\sn)_{\scriptscriptstyle +}$. Let $U_1,U_2:\mathbb{R}_{\scriptscriptstyle +}\times \sn\to M(\sn)_{\scriptscriptstyle +}$ be $\mu$-strong Carath\'eodory and uniformly controlled maps with $U_1(0,s)=U_2(0,s)=0$ for $\mu$-almost every $s\in\sn$. The following are equivalent: 
\begin{itemize}
    \item For every $\lambda\geq 0$, $U_1(\lambda,s)=U_2(\lambda,s)$ for $\mu$-almost every $s\in\sn$.
    \item The valuations defined through $U_1$ and $U_2$ coincide. That is, for every $f\in C(\sn)_{\scriptscriptstyle +}$ and $g\in C(\sn)$ we have
    \[
        \int_\sn \langle U_1(f(s),s),g\rangle\d \mu(s)=\int_\sn \langle U_2(f(s),s),g\rangle \d \mu(s).
    \]    
\end{itemize}
\end{corollary}

\begin{proof}
Assume that for every $\lambda\geq0$, there exists $S_\lambda\subset\sn$ of $\mu$-full measure such that if $s\in S_\lambda$, then $U_1(\lambda,s)=U_2(\lambda,s)$. For every $A_1,\dots,A_m\in\Sigma_n$, every $\lambda_1,\dots,\lambda_m\geq 0$ and every $g\in C(\sn)$ one has
\begin{multline*}
    \left\langle \overline{\oV}_1\left(\sum_{i=1}^m\lambda_i \chi_{A_i}\right),g\right\rangle=\int_{\sn}\left\langle U_1\left(\sum_{i=1}^m\lambda_i \chi_{A_i},s\right),g\right\rangle\d \mu(s)\\
    \hspace{20mm}=\sum_{i=1}^m\int_{A_i}\langle U_1(\lambda_i,s),g\rangle \d \mu(s)=\sum_{i=1}^m\int_{A_i}\langle U_2(\lambda_i,s),g\rangle \d \mu(s)\\
    =\left\langle \overline{\oV}_2\left(\sum_{i=1}^m\lambda_i \chi_{A_i}\right),g\right\rangle\hspace{40mm}
\end{multline*}
Therefore, $\overline{\oV}_1$ and $\overline{\oV}_2$ coincide on simple functions, which form a dense set in $B(\sn)$. By Corollary \ref{co: extension}, $\overline{\oV}_1$ and $\overline{\oV}_2$ are weak$^*$ continuous and coincide in a dense set, so we must have $\oV_1=\oV_2$.

Conversely, if $\oV_1=\oV_2$, then $\overline{\oV}_1=\overline{\oV}_2$ and for all $\lambda\geq0$, $A\in\Sigma_n$ and $p\in\poly$,
\[
    \int_{A}\langle U_1(\lambda,s),p\rangle\mathrm{d} \mu(s)=\langle \overline{\oV}_1(\lambda\chi_A),p\rangle =\langle \overline{\oV}_2(\lambda\chi_A),p\rangle=\int_A  \langle U_2(\lambda,s),p\rangle \mathrm{d} \mu(s).
\]
Fix $\lambda\geq 0$. For every $p\in\poly$, there exists a set $S_p\subset \sn$ of $\mu$-full measure such that if $s\in S_p$, then
\[
    \langle U_1(\lambda,s),p\rangle=\langle U_2(\lambda,s),p\rangle.
\]
Consider $\bigcap_{p\in\poly}S_p$ which is of $\mu$-full measure. For every $s$ in this set, 
\[
    \langle U_1(\lambda,s),p\rangle=\langle U_2(\lambda,s),p\rangle \quad \text{for all }p\in\poly.
\]
Since $\poly$ is a dense set in $C(\sn)$, it follows that if $s\in \bigcap_{p\in\poly}S_p$, then $U_1(\lambda,s)=U_2(\lambda,s)$.
\end{proof}
\goodbreak

\section{Rotation Equivariant, Homogeneous and locally determined valuations}\label{S:equivariant}

In this section, we specialize our main result to obtain classification theorems for valuations with certain geometric properties.

\subsection{Rotation equivariant valuations}\label{ss_rev}
Our first aim is to show that the integral representation of Theorem \ref{th:representation} has a special form in the rotation equivariant case. In particular, many of the technical steps 
in the previous sections can be simplified in this setting. 

Recall the definition of $\nu_{\lambda,g}$ in \eqref{eq:defnulambda} and the definition of the control measure $\mu$ in \eqref{eq_deK_mu}. First, the rotation equivariance implies that for every $\phi\in\on$ and every $g\in C(\sn)$
\begin{equation}\label{eq:rotation}
    \nu_{\lambda,g}(A)=\langle \overline{V}(\lambda\chi_A),g\rangle=\langle \overline{V}(\lambda\chi_{\phi A}),\phi g\rangle=\nu_{\lambda,\phi g}(\phi A).
\end{equation}
    
In particular, if $g=\uno$, then $\nu_{\lambda,\uno}$ is a rotation invariant measure. This will also be the case for the control measure $\mu$.

\begin{lemma}\label{l:invarian_mu}
Let $V\colon C(\sn)_{\scriptscriptstyle +}\to M(\sn)_{\scriptscriptstyle +}$ be a weak$^*$ continuous valuation that is rotation equivariant and satisfies $\oV(0)=0$. The control measure $\mu$ defined in \eqref{eq_deK_mu} is rotation invariant, that is, for every $A\in\Sigma_n$ and every $\phi\in\on$ one has 
\[
    \mu(A)=\mu(\phi A).
\]
Therefore, $\mu$ is a multiple of the Hausdorff measure $\mathcal{H}^{n-1}$ on $\sn$.
\end{lemma}
\begin{proof}
We recall that $\mu$ is defined in \eqref{eq_deK_mu} by a series which elements are $\mu_{i,\uno}$ with $i\in\mathbb{N}$. For every open set $G\subset\sn$, these are defined by \eqref{eq:control measure} as 
\[
    \mu_{i,\uno}(G)=\sup\left\{V_{\uno} (f): f\prec G,\|f\|_{\scriptscriptstyle \infty} \leq i \right\}.
\]
For $\phi\in\on$ and every open set $G\subset \sn$, since  $\|\phi f\|_{\scriptscriptstyle \infty}=\|f\|_{\scriptscriptstyle \infty}$, we have
\begin{align*}
    \mu_{i,\uno}(\phi G)&=\sup\left\{V_{\uno} (f): f\prec \phi G,\|f\|_{\scriptscriptstyle \infty} \leq i \right\}\\&=\sup\left\{V_{\uno} (f): \phi^{-1}f\prec G,\|f\|_{\scriptscriptstyle \infty} \leq i \right\}\\
    &=\sup\left\{V_{\uno} (f): \phi^{-1}f\prec G,\|\phi^{-1}f\|_{\scriptscriptstyle \infty} \leq i \right\}.
\end{align*}
Therefore, taking $f'=\phi^{-1}f$, using that $\phi$ is bijective and $\phi^{-1}\uno=\uno$, we get
\begin{align*}
    \mu_{i,\uno}(\phi G)&=\sup\left\{V_{\uno} (\phi f'): f'\prec G,\|f'\|_{\scriptscriptstyle \infty} \leq i \right\}\\
    &=\sup\left\{V_{\uno} ( f'): f'\prec G,\|f'\|_{\scriptscriptstyle \infty} \leq i \right\}\\
    &=\mu_{i,\uno}(G).
\end{align*}
Thus, by regularity of the measures it follows that $\mu_{i,\uno}$, and consequently $\mu$, are rotation invariant.
\end{proof}

The main advantage in this case is the fact that the function $K_g(\lambda,s)$ from Lemma \ref{l:weak pseudokernel} can be replaced with another function $K_g'(\lambda,s)$ with a more explicit description. This new function leads to a  
measure $U'(\lambda,s)$ which satisfies the main Theorem \ref{th:representation} and is rotation equivariant. In order to construct this $U'$, let $p\in\poly$, let $K_p(\lambda,s)$ denote the function given by Lemma \ref{l:weak pseudokernel} and let $S$ be the set defined in \eqref{eq_set}.

From the Lebesgue Differentiation Theorem (see \cite[2.8-2.9]{Federer} for an extensive discussion), there exists a Borel set $S_{\lambda,p}\subset \sn$ such that $\mu(\sn\setminus S_{\lambda,p})=0$ and 
\[
    K_{p}(\lambda,s)=\lim_{\rho\to 0}\frac{\nu_{\lambda,p}(B(s,\rho))}{\mu(B(s,\rho))}
\] 
for every $s\in S_{\lambda,p}$, where $B(s,\rho)$ denotes the closed ball centered at $s$ with radius $\rho$. We will see that the limit is actually well-defined for every $s\in\sn$ and every $g\in C(\sn)$ and, therefore, we will define $\langle U'(\lambda,s),g\rangle$ as that limit. 

First, define $S_\lambda=
\bigcap_{p\in\poly}S_{\lambda,p}$ and let $g\in C(\sn)$. Note that
\begin{equation}\label{eq:U=U'ae}
    \lim_{\rho\to 0}\frac{\nu_{\lambda,g}(B(s,\rho))}{\mu(B(s,\rho))}=\langle U(\lambda,s),g\rangle
\end{equation}
for all $s\in S_\lambda$. Indeed, let $\{p_i\}_{i=1}^\infty\subset\poly$ be a sequence that converges to $g$, then
\begin{align*}
    \left|\limsup_{\rho\to 0}\frac{\nu_{\lambda,g}(B(s,\rho))}{\mu(B(s,\rho))} - \langle U(\lambda,s),g\rangle\right|\hspace{65mm}\\
    =\left|\limsup_{\rho\to 0}\frac{\nu_{\lambda,g}(B(s,\rho))}{\mu(B(s,\rho))} -\lim_{i\to\infty}\langle U(\lambda,s),p_i\rangle\right| \hspace{16mm} \\
    =\left|\limsup_{\rho\to 0}\frac{\nu_{\lambda,g}(B(s,\rho))}{\mu(B(s,\rho))} -\lim_{i\to\infty}\limsup_{\rho\to 0}\frac{\nu_{\lambda,p}(B(s,\rho))}{\mu(B(s,\rho))}\right|\hspace{1mm}\\
    =\left|\lim_{i\to\infty}\limsup_{\rho\to 0}\frac{\nu_{\lambda,g-p_i}(B(s,\rho))}{\mu(B(s,\rho))} \right|\leq \lim_{i\to\infty} \|g-p_i\|\zeta (\lambda).
\end{align*}
The limit inferior is treated analogously. Therefore, the limit exists for any $g\in C(\sn)$ and any $s\in S_\lambda$. To extend beyond $S_\lambda$, note that, since $S_\lambda$ is a $\mu$-full measure set, $S_\lambda\neq\emptyset$. If $t\not\in S_\lambda$, then there exist $\phi\in\on$ and $s\in S_\lambda$ such that $\phi(t)=s$. By \eqref{eq:rotation} and Lemma \ref{l:invarian_mu}, we have
\[
    \frac{\nu_{\lambda,g}(B(t,\rho))}{\mu(B(t,\rho))}=\frac{\nu_{\lambda,\phi g}(\phi( B(t,\rho)))}{\mu(B(t,\rho))}=\frac{\nu_{\lambda,\phi g}(B(s,\rho))}{\mu(B(s,\rho))}.
\]
Since the right hand side converges when $\rho$ goes to 0, the left hand side also converges. Thus, for any $s\in \sn$ and any $g\in C(\sn)$, we define 
\begin{equation}\label{eq:defin_U'}
    U'(\lambda,s)(g)=
    \lim_{\rho\to 0}\frac{\nu_{\lambda,g}(B(s,\rho))}{\mu(B(s,\rho))}.
\end{equation}
Clearly, the linearity of $g\mapsto \nu_{\lambda,g}(A)$ for any $A\in\Sigma_n$ together with the control property implies that $U'$ is a linear and continuous operator defined on $C(\sn)$, that is, a measure. Moreover, by \eqref{eq:U=U'ae} and Corollary \ref{co:uniqueness}, we get that $U'$ and $U$ represent the same valuation $V$. The only property left to prove is the continuity on $\lambda$.

\begin{lemma}
Let $g\in C(\sn)$ and $s\in\sn$. The mapping
\[
    \lambda\mapsto \langle U'(\lambda,s),g\rangle
\]
is continuous on $\rp$.
\end{lemma}
\begin{proof}
Let $\lambda\in\rp$ and let $\{\lambda_m\}_{m\in\mathbb{N}}\subset\rp$ be a sequence such that $\lambda_m\to \lambda$. Consider the set $S_0
=\bigcap_{m\in\mathbb{N}}S_{\lambda_m}\cap S_\lambda$. If $t\in S_0
$ and $g\in C(\sn)$, then by \eqref{eq:U=U'ae}
\[
    \langle U(\lambda,t),g\rangle=\langle U'(\lambda,t),g\rangle \quad \text{and}\quad \langle U(\lambda_m,t),g\rangle=\langle U'(\lambda_m,t),g\rangle
\]
for all $m\in\mathbb{N}$. Therefore, by Lemma \ref{l:Uwcont} it is continuous.
Note that $S_0
$ is of $\mu$-full measure and, therefore, $S_0\neq\emptyset$. Thus, if $t\not\in S_0
$, consider $s\in S_0$ and $\phi\in\on$ such that $\phi(t)=s$. By \eqref{eq:rotation} and Lemma \ref{l:invarian_mu}, we have
\begin{align*}
    \langle U'(\lambda,t)-U'(\lambda_m,t),g\rangle &=\lim_{\rho\to 0}\frac{\nu_{\lambda,g}(B(t,\rho))-\nu_{\lambda_m,g}(B(t,\rho))}{\mu(B(t,\rho))}\\
    &=\lim_{\rho\to 0}\frac{\nu_{\lambda,\phi g}(\phi(B(t,\rho)))-\nu_{\lambda_m,\phi g}(\phi(B(t,\rho)))}{\mu(B(t,\rho))}\\
    &=\lim_{\rho\to 0}\frac{\nu_{\lambda,\phi g}(B(s,\rho))-\nu_{\lambda_m, \phi g}(B(s,\rho))}{\mu(B(t,\rho))}\\
    &=\langle U'(\lambda,s)-U'(\lambda_m,s),\phi g\rangle.
\end{align*}
\end{proof}

In the following subsections, whenever we assume that the valuation is rotation equivariant we will be working with the specific map $U'$.

\begin{lemma}\label{l: equivariant measures}
For every $\lambda \in \mathbb{R}_{\scriptscriptstyle +}$, the mapping $(s,B)\mapsto U'(\lambda,s)(B)$, defined on $\sn\times \Sigma_n$, is rotation equivariant, that is, for every $\phi\in\on$, for every $s\in\sn$ and every $B\in\Sigma_n$, we have
\[
    U'(\lambda,s)(B)=U'(\lambda,\phi (s))(\phi B).
\]
In particular, if $\phi\in\on$ leaves $s\in S^{n-1}$ invariant, then 
\[
    U'(\lambda,s)(B)=U'(\lambda,s)(\phi B).
\]
\end{lemma}

\begin{proof}
Let $s\in \sn $, $\lambda\in\mathbb{R}_{\scriptscriptstyle +}$ and $\phi \in \on$. By \eqref{eq:rotation} and \eqref{eq:defin_U'}, for $g\in C(\sn )$, we have 
\begin{equation}\label{eq:U equivariant g}
\begin{split}
    \langle U'(\lambda,s),g\rangle &= \lim_{\rho\rightarrow 0} \frac{ \nu_{\lambda,g}( B(s, \rho))}{\mu(B(s,\rho))} \\
    &= \lim_{\rho\rightarrow 0} \frac{\nu_{\lambda,\phi g}( B(\phi(s), \rho))}{\mu(B(\phi(s),\rho))} =  \langle U'(\lambda,\phi(s)),\phi\, g\rangle.
\end{split}
\end{equation}
Now, for $B\in \Sigma_n$, by the regularity of the measures $U'(\lambda,s)$ and $U'(\lambda,\phi(s))$, given $\varepsilon>0$, there exist compact sets $K_1\subset B$ and $K_2\subset \phi B$ such that 
$$
U'(\lambda,s)(B)\le U'(\lambda,s)(K_1)+\varepsilon,\quad \quad U'(\lambda,\phi(s))(\phi B)\le U'(\lambda,\phi(s))(K_2)+\varepsilon.
$$
Let $K=K_1\cup\phi^{-1}(K_2)$, and take a decreasing sequence $(f_j)_{j\in\mathbb N}$ in $\cfn$  such that $\|f_j\|_{\scriptscriptstyle \infty}\leq1$, $f_j(t)=1$ for $t\in K$ and $f_j(t)=0$ for $\dist(t,K)\geq 1/j$. By the Monotone Convergence Theorem, it follows that
$$
U'(\lambda,s)(K)=\lim_{j\to\infty}\langle U'(\lambda, s),f_j\rangle,\hspace{3.5mm} U'(\lambda,\phi(s))(\phi K)=\lim_{j\to\infty}\langle U'(\lambda,\phi (s)),\phi f_j\rangle.
$$
Hence, using \eqref{eq:U equivariant g}, 
\begin{align*}
U'(\lambda,s)(K_1) &\leq U'(\lambda,s)(K)=\lim_{j\to\infty}\langle U'(\lambda,s),f_j\rangle\\
&=\lim_{j\to\infty}\langle U'(\lambda,\phi (s)),\phi f_j\rangle=U'(\lambda,\phi(s))(\phi K).
\end{align*}
Thus, we have $U'(\lambda,s)(B)\le U'(\lambda,\phi(s))(\phi B) +\varepsilon$.
Similarly, we get $$U'(\lambda,\phi(s))(\phi B)\le U'(\lambda,s)(B)+\varepsilon.$$
Since $\varepsilon>0$ was arbitrary, this shows that $U'$ is rotation equivariant.
\end{proof}

We will make use of these properties to provide the following natural decomposition of $U'(\lambda,s)$: Given $\pole\in S^{n-1}$, we can decompose $\rn = \mathbb R \pole \times \pole^\perp $, where $\pole^\perp$ denotes the hyperplane orthogonal to $\pole$. This decomposition induces a bijection $\vartheta_{\pole}\colon  \sn \to  (-1,1)\times S^{n-2} \cup \{(1,0,\ldots,0)\}\cup \{(-1,0,\ldots,0)\}$, defined by 
\begin{equation}\label{e_bijection}
    (s_1, \ldots, s_n) \mapsto
    \begin{cases}
       \Big(s_1,\displaystyle \frac{s_2}{\sqrt{1-s_1^2}},
        \ldots,\frac{s_{n}}{\sqrt{1-s_1^2}} \Big) &\text{ for } s_1\ne \pm1\\
        (\pm 1,0,\ldots,0) &\text{ for } s_1=\pm 1,
    \end{cases}
\end{equation}
where $s_1, \dots, s_{n}$ are the coordinates of the point $s\in\sn$ with respect to an orthonormal basis $\{e,e_2,\ldots,e_n\}$. The map $\vartheta_{\pole}$ clearly defines a homeomorphism between $\sn\backslash\{\pm \pole\}$ and the cylinder $(-1,1) \times \sn$.

In the following, we denote the normalized Hausdorff measures with respect to the sphere as $\overline{\mathcal{H}}^k$, that is, for any $0\leq k\leq n$ (with $k$ integer), we set
\[
    \overline{\mathcal{H}}^{k}(A)=\frac{\mathcal{H}^{k}(A)}{\mathcal{H}^{k}(S^{k})}
\]
for any Borel set $A\subset S^k$.

\begin{lemma}\label{l:haar n-2}
For every $\lambda \in \rp$, there exists a measure $\sigma_\lambda$ on the Borel sets in $(-1,1)$ such that for every $\pole\in \sn$ 
$$U'(\lambda, \pole)\circ \vartheta^{-1}_{\pole} = \sigma_\lambda\otimes \overline{\mathcal{H}}^{n-2}$$
on $(-1,1)\times S^{n-2}$ with $\vartheta_{\pole}$ defined by \eqref{e_bijection}. 
\end{lemma} 
\begin{proof}
Let $C \times D$ be a rectangle in $ (-1,1)\times S^{n-2}$. Lemma~\ref{l: equivariant measures} implies that 
$$
(U'(\lambda, \pole)\circ\vartheta_{\pole}^{-1})(C\times D)) = (U'(\lambda, \pole)\circ\vartheta_{\pole}^{-1})(C\times \tilde{\phi} \, D)
$$
for every  $\tilde{\phi}\in \om$.
Therefore, for every Borel subset $C\subset (-1,1)$, the measure on $S^{n-2}$ defined by $D\mapsto (U'(\lambda, s)\circ\vartheta_{\pole}^{-1})(C\times D)$ is a constant multiple of $\overline{\mathcal{H}}^{n-2}$, say, $\sigma_\lambda(C) \,\overline{\mathcal{H}}^{n-2}$. It is now easy to see that the map $C\mapsto \sigma_\lambda(C)$ defines a Radon measure on $(-1,1)$. 

Finally, to prove that this measure is independent of the point, let $\pole$ and  $\pole'$ be distinct points of $S^{n-1}$. Choose $\psi\in \on$ so that $\psi(\pole)=\pole'$. For $\lambda\in\R$ and a Borel set $A\subset (-1,1)$, it follows from Lemma \ref{l: equivariant measures} that 
\begin{align}
U'(\lambda, \pole)(\vartheta^{-1}_\pole (A\times S^{n-2})) &= U'(\lambda, \psi(\pole))(\psi\vartheta^{-1}_\pole (A\times S^{n-2}))) \\
&= U'(\lambda, \pole')(\vartheta^{-1}_{\pole'}( A\times S^{n-2})). 
\end{align}
\end{proof}
\goodbreak
This lemma shows that, in the rotation equivariant case, the push-forward of $U'(\lambda, \pole)\in \measn$ by $\vartheta_\pole$ is an element of $M((-1,1)\times S^{n-2} \cup \{\pole\} \cup \{-\pole\})_{\scriptscriptstyle +}$ of the form 
\begin{equation}\label{eq:sigma}
\sigma_\lambda \otimes \overline{\mathcal{H}}^{n-2}    + a_\lambda \,\delta_{\pole} + b_\lambda \delta_{-\pole} 
\end{equation}
with $a_\lambda, b_\lambda \in \rp$. In particular, for every $g\in C(\sn)$
\begin{equation}
    \langle U'(\lambda,s),g\rangle =\int_{(-1,1)}\int_{\sm}g(\vartheta_s^{-1}(\alpha,t))\d \overline{\mathcal{H}}^{n-2}(t) \d \sigma_\lambda (\alpha) + a_\lambda g(s)+b_\lambda g(-s).
\end{equation}
Although this representation can be useful, one cannot ensure continuity in $\lambda$ of the individual terms. For instance, $a_\lambda$ need not be continuous in $\lambda$. Moreover, it is not clear which conditions on $\sigma_\lambda$, $a_\lambda$ and $b_\lambda$ should be imposed in order for them to define a weak$^*$ valuation. To obtain a complete characterization, we consider the three terms together as a measure on the interval $[-1,1]$. To this end, for each $s\in\sn$, we introduce the following operator from $C(\sn)$ to $C([-1,1])$:
\[
    \radon g(s,\alpha)=\left\{\begin{array}{ll}
        g(s) &\text{if }\alpha=1, \\
        g(-s) &\text{if }\alpha=-1, \\
        \displaystyle\int_{\sm}g(\vartheta_s^{-1}(\alpha,t))\d\overline{\mathcal{H}}^{n-2}(t) &\text{if }\alpha\in(-1,1). \\
    \end{array} \right.
\]
Clearly, the previous mapping is well-defined (the image is a continuous function), is linear, continuous and surjective. Note that it is a normalized Radon transform.

\begin{theorem}\label{th:representation equivariant}
The map $\oV\colon \cfn\to \measn$ is a weak$^*$~continuous and rotation equivariant valuation with $\oV(0)=0$ if and only if there exists a weak$^*$~continuous map $\rhot\colon \rp\to M([-1,1])_{\scriptscriptstyle +}$  with $\rhot_0=0$ such that
\[
\langle V (f),g\rangle= \int_\sn \int_{[-1,1]}\radon g(s,\alpha)\d\rhot_{f(s)}(\alpha)\d\hn(s) 
\]
for any $f\in \cfn$ and $g\in C(\sn)$.
\end{theorem}
\begin{proof}
First, we will define $\rhot_\lambda$ by duality. For any $\G\in C([-1,1])$ we define
\[
    \rhot_\lambda(\G) =\int_{(-1,1)} \G(\alpha)\d\sigma_\lambda(\alpha)+ a_\lambda     \G(1)+     b_\lambda \G(-1)
\]
using \eqref{eq:sigma}.
Clearly, $\rhot_\lambda$ is linear and bounded and, therefore, $\rhot_\lambda$ is a measure on $[-1,1]$. Moreover,
\begin{align*}
    \langle \rhot_\lambda,\radon g (s, \cdot)\rangle &=\int_{(-1,1)} \left(\int_{\sm}g(\vartheta_s^{-1}(\alpha,t))\d \overline{\mathcal{H}}^{n-2}(t)\right)\d\sigma_\lambda(\alpha)\\
    &\hspace{5mm}+a_\lambda g(s)+b_\lambda g(-s)\\
    &=\langle U'(\lambda,s),g\rangle.
\end{align*}
Now, we have to check that $\rhot_\lambda $ is weak$^*$ continuous but this is a consequence of two facts, the surjectivity of $g\mapsto \radon g (s,\cdot)$ and the weak$^*$ continuity of the mapping $\lambda\mapsto U'(\lambda,s)$. Indeed, let $\G\in C([-1,1])$, $\lambda\in\rp$ and $\{\lambda_m\}_{m\in\mathbb{N}}\subset \rp$ a sequence that converges to $\lambda$. Since $g\mapsto \radon g (s,\cdot)$ is a surjective mapping, there exists $g\in C(\sn)$ such that $\radon g(s, \cdot)=\G$ and, thus,
\[
    \langle \rhot_{\lambda_m},\G\rangle=\langle \rhot_{\lambda_m},\radon g (s, \cdot)\rangle =\langle U'(\lambda_m,s),g\rangle\to \langle U'(\lambda,s),g\rangle=\langle \rhot_\lambda,\G\rangle
\]
as $m\to\infty$.

\goodbreak
Conversely, let $\rhot_\lambda\in M([-1,1])_{\scriptscriptstyle +}$. For any $\lambda\in\rp$ and any $s\in\sn$, define the mapping
\[
     U'(\lambda,s)(g) =\int_{[-1,1]} \radon g (s, \alpha)\d \rhot_\lambda (\alpha)=\langle \rhot_\lambda, \radon g (s, \cdot)\rangle.
\]
Clearly, this mapping is linear and continuous; therefore, it defines a measure. Observe that that $\|U'(\lambda,s)\|\leq \|\rhot_\lambda\|$; hence, it is uniformly controlled. We now verify that it is $\mathcal{H}^{n-1}$-strong Carath\'eodory. First, it follows easily that the mapping
\[
    \lambda\mapsto \langle U'(\lambda,s),g\rangle
\]
is continuous for any $g\in C(\sn)$ and any $s\in \sn$. For any $\alpha\in(-1,1)$ and $e\in\sn$, define
\[
    C_\alpha=
    \frac{1}{\hm(\{s\in\sn:\langle s,e\rangle=\alpha\})}.
\] 
Fix $g\in C(\sn)$ and $\lambda\in\rp$. Let $s\in \sn$ and let $\{s_m\}_{m\in\mathbb{N}}\subset\sn$ be a sequence that converges to $s$. Then
\[
    \radon g (s_m, \pm 1)=g(\pm s_m)
\]
and, for $\alpha\in(-1,1)$,
\begin{align*}
   \radon g (s_m, \alpha)&=\int_{\sm}g(\vartheta^{-1}_{s_m}(\alpha,t))\d\overline{\mathcal{H}}^{n-2}(t)\\
    &=C_\alpha \int_{\{r\in\sn:\langle r,s_m\rangle=\alpha\}}g(r)\d\hm(r).
\end{align*} 
Since $g$ is uniformly continuous, it follows that $\radon g (s_m, \cdot)$ converges uniformly to $\radon g (s, \cdot)$. Hence, the mapping 
\[
    s\mapsto \langle U'(\lambda,s),g\rangle
\]
is continuous and, in particular, it is Borel measurable. Thus, by Theorem \ref{th:representation}, the mapping $f\mapsto \oV (f)$, given by
\begin{align*}
    \langle \oV(f),g\rangle &=\int_\sn \langle U'(\lambda,s),g\rangle \d \hn(s)\\
    &=\int_{\sn}\int_{[-1,1]} \radon g (s, \alpha)\d\rhot_\lambda(\alpha)\d\hn(s),
\end{align*}
is a weak$^*$ continuous valuation.

Finally, let us check that $\oV$ is rotation equivariant. Note that it is enough to see that $\radon (\phi g) (\phi s, \cdot)= \radon g (s, \cdot)$ for any $g\in C(\sn)$ and any $\phi\in\on$. This follows easily since
\[
   \radon (\phi g) (\phi s, \pm 1) =\phi g(\phi(\pm s))=g(\pm s),
\]
and for any $\alpha\in(-1,1)$ we have
\begin{align*}
   \radon (\phi g) (\phi s, \alpha)&=\int_{\sn}\phi g(\vartheta_{\phi(s)}^{-1}(\alpha,t))\d \overline{\mathcal{H}}^{n-2}(t)\\
    &=C_\alpha \int_{\{r\in\sn: \langle r,\phi(s)\rangle=\alpha\}}\phi g(r)\d \hm(r)\\
    &=C_\alpha \int_{\{r\in\sn:\langle s,r\rangle =\alpha\} }g(r)\d\hm(r)=\radon g (s, \alpha).
\end{align*}
Therefore, the valuation $\oV$ is rotation equivariant.
\end{proof}

\noindent
Theorem \ref{th:rep equivariant} follows immediately from this result.

\subsection{Homogeneous valuations}
We derive next the following corollary to Theorem \ref{th:representation}. Let $q\in\R$.

\begin{corollary}\label{c:hom}
Let $\,\oV\colon \cfn\to \measn$ be a weak$^{\,*}$\! continuous and \mbox{$q$-homogeneous} valuation. The map $U\colon \rp \times \sn\!  \to \measn$ from Theorem \ref{th:representation} can be chosen so that $$U(\lambda, s)=\lambda^q\, U(1,s)$$
for every $\lambda> 0$ and  $s\in \sn $. Conversely, if $\,U(\lambda, s)=\lambda^q\, U(1,s)$ for every $\lambda> 0$ and  $s\in \sn $, then the associated valuation is $q$-homogeneous. 
\end{corollary} 

\begin{proof}
First, if $q=0$, then weak$^*$ continuity implies that $\oV(f)=\oV(0)$ for every $f\in\cfn$. It is easy to see that the map $U(\lambda,s)=\frac{1}{\mathcal{H}^{n-1}(\sn)}\oV(0)$ represents the valuation with a similar integral representation as in Theorem \ref{th:representation}, even if $U(0,s)$ may not be 0 $\mu$-almost everywhere. It is also clear that this map satisfies $U(\lambda,s)=\lambda^q\,U(1,s)$

If $q\neq 0$, then $q$-homogeneity implies that $\oV(0)=\lambda^q\oV(0)$ for every $\lambda>0$, and therefore $\oV(0)=0$ if $q\neq 0$. Thus, we can apply Theorem \ref{th:representation}, and there exists a map $U$ as in the statement. Let us consider the extension $\overline{\oV}$ of $\oV$  as defined in Corollary \ref{co: extension} and let us check that $\overline{\oV}$ is $q$-homogeneous. To this end, let $C \subset \sn $ be closed and $\lambda>0$. For $\tau>0$, let $f_\tau\in C(\sn )$ be such that $f_\tau(s)=1$ for every $s\in C$, $\|f_\tau\|_{\scriptscriptstyle \infty}\leq 1$ and $f_\tau(s)= 0$ whenever $\dist(s, C)\geq \tau$. By Corollary \ref{co: extension}, we have  $$\langle \overline{\oV}(\lambda\chi_C),g\rangle= \lim_{\tau\rightarrow 0} \langle \oV(\lambda f_\tau),g\rangle = 
\lim_{\tau\rightarrow 0} \lambda^q \langle \oV( f_\tau),g\rangle = \langle \lambda^q\, \overline{\oV}(\chi_C),g\rangle$$
for every $g\in C(\sn )$ and $\lambda\in [0,\lambda_0]$. Hence, $\overline{\oV}(\lambda\chi_C)=\lambda^q\overline{\oV}(\chi_C)$.

Now, we claim that the class of subsets $A$ of $\sn $  for which $\overline{\oV}(\lambda \chi_A) = \lambda^q\overline{\oV}(\chi_A)$ for every $\lambda \ge 0$ is a $\sigma$-algebra. Indeed, using the valuation property of $\overline{\oV}$,  we see that this class is closed under finite union and taking complementary sets. 
It remains to show that it is closed under countable unions. Let $(A_i)_{i\in \mathbb N}$  be a disjoint sequence of subsets in this class. For $A=\bigcup_i A_i$, the
 sequence $\sum_{i=1}^m \chi_{A_i}$ converges almost everywhere to $\chi_A$. By Corollary \ref{co: extension}, we have 
\begin{align}
\langle \overline{\oV}(\lambda\chi_A),g\rangle &= \lim_{m\to\infty} \langle \overline{\oV}(\lambda \sum\nolimits_{i=1}^m \chi_{A_i}),g\rangle\\
&= \lim_{m\to\infty} \lambda^q \langle\overline{\oV}( \sum\nolimits_{i=1}^m \chi_{A_i}) ,g\rangle 
= \langle \lambda^q \overline{\oV}(\chi_A),g\rangle
\end{align}
for every $g\in C(\sn )$.
\goodbreak

We consider the map  $U\colon \rp \times \sn \to \measn$ associated to $\oV$. 
For $s\in \sn $ and $\lambda\ge 0$, we define the map 
$U'\colon\rp\times \sn \to \R$ by 
$$
\tilde U(\lambda, s)=\lambda^q U(1,s).
$$

It is clear that $\tilde U$ is a $\mu$-strong Carath\'eodory and uniformly controlled map. Now, for every $g\in C(\sn)$ and every $A\in \Sigma_n$ one gets that
\[
    \int_A \!\langle U(\lambda,s),g\rangle\d\mu(s)=\langle \overline{\oV}(\lambda\chi_A),g\rangle=\langle \lambda^q \overline{\oV}(\chi_A),g\rangle=\lambda^q\int_A\! \langle U(1,s),g\rangle\d\mu(s).
\]
This implies that for every $\lambda\in\mathbb{R}_+$, there exists a set $S'$ of $\mu$-full measure where $U(\lambda,s)=\tilde U(\lambda,s)$ for every $s\in S'$. By Corollary \ref{co:uniqueness}, $U$ and $\tilde U$ define the same valuation. The converse is clear. 
\end{proof}

\goodbreak

\subsection{Locally determined valuations and dual area measures}\label{ss:ld}

We call a valuation $\oV\colon \cfn\to \measn$  {\em locally determined} if  
\begin{equation}\label{eq:locallydet}
f_1\wedge \chi_B = f_2\wedge \chi_B \,\text{ implies }\, \oV(f_1)(B)=\oV(f_2)(B)    
\end{equation} 
for every open set $B\subset \sn$ and  $f_1, f_2 \in \cfn$. In fact, an argument similar to \cite{Jonas2026} shows that a locally determined map that is weak$^*$ continuous is actually a valuation in this context.

Note in particular that locally determined maps satisfy the following 
\begin{equation}\label{eq_ld2}
    f\wedge \chi_B = 0 \,\text{ implies }\, \oV(f)(B)=V(0)(B)=0
\end{equation}
for every open set $B\subset \sn$ and  $f \in C(\sn )$. We will use \eqref{eq_ld2} in the proof of the following lemma. 

\begin{lemma}\label{l:locally determined}
The map $\,\oV\colon \cfn\to \measn$ is a weak$^*$~continuous, locally determined, rotation equivariant valuation with $\oV(0)=0$ if and only if there exists a continuous function $\theta\colon\rp\to \rp$ with $\theta(0)=0$ such that for every $\lambda \in \R_+$ and  $s\in \sn $
\[
U'(\lambda, s) = \theta(\lambda)\,\delta_s.
\]
\end{lemma}

\begin{proof}
Let $s\in \sn $ and $\lambda\ge 0$. Suppose $g\in C(\sn )$  is a function such that $s\not \in \supp (g)$, the support of $g$. Since $\supp (g)$ is a closed set, 
$$
B(s, \rho)\cap \supp (g) = \emptyset
$$ 
for every $0<\rho<\dist(s, \supp (g))$. 
In particular, $\langle \oV(\lambda f),g\rangle=0$ for every $f\prec B(s,\rho)$. 
Therefore, we can consider the map $U'$ constructed in Section \ref{ss_rev}, and we have
$$
U'(\lambda, s)(g) = \lim_{\rho\rightarrow 0} \frac{\langle \oV(\lambda B(s,\rho)), g\rangle}{\mu(B(s,\rho))} = 0.
$$
Since this holds for every  $g\in C(\sn )$ with $s\not \in \supp (g)$, it follows that $U'(\lambda, s) = u(\lambda,s) \,\delta_s$ for some $u(\lambda,s)\in \rp$. 

Let now $s'\not = s$ be another point in $\sn $. Choose $\phi\in \on$ such that $\phi(s)=s'$. 
By Lemma \ref{l: equivariant measures}, we have $$u(\lambda, s) = U'(\lambda, s)(\uno) = U'(\lambda, s')(\uno)= u(\lambda, s').$$ Therefore, $u(\lambda, s)$ does not depend on $s$, and we can rename it as $\theta(\lambda)$. 

To see that $\theta\colon \rp\to\rp$ is continuous, just note  that
$$\langle \oV(\lambda \uno),\uno\rangle= \int_{\sn } \langle U'(\lambda, s),\uno\rangle \d\mu(s) = \int_{\sn } \theta(\lambda) \d\mu(s) =\mu(\sn) \theta(\lambda). $$ 
Hence, the continuity of $\theta$ follows from the weak$^*$ continuity of $\oV$ and from $\theta(0)=0$.

The converse statement is immediate. 
\end{proof}

The following classification result is an immediate consequence of Theorem \ref{th:representation equivariant} and Lemma~\ref{l:locally determined}. 

\begin{theorem}\label{t:orlicz} 
A map $\oV\colon  \cfn\to \measn$ is a weak$^*$\! continuous, locally determined, rotation equi\-variant valuation with $\oV(0)=0$ if and only if there is a continuous function $\theta\colon \rp\to\rp$ with $\theta(0)=0$ such that 
$$\langle \oV(f),g\rangle =\int_\sn \theta(f(s)) g(s) \d\hn(s)$$
for every $f\in\cfn$ and $g\in C(\sn)$.
\end{theorem}

\noindent
Requiring that the valuation is, in addition, $q$-homogeneous with $q\in\R$, by Corollary \ref{c:hom} and the above we obtain the following classification.

\begin{theorem}\label{t:characterization}
A map $\oV\colon  \cfn\to \measn$ is a weak$^*$\! continuous, locally determined, rotation equi\-variant, and $q$-homogeneous valuation if and only if there is $c\ge 0$ such that
$$\langle \oV(f),g\rangle =c\,\int_\sn f^q(s) g(s) \d\hn(s)$$
for every $f\in\cfn$ and $g\in C(\sn)$.
\end{theorem}

Finally, we return to valuations on star bodies in $\R^n$. We call a valuation $\oZ\colon \starn\to \measn$ \emph{locally determined} if the valuation $f\mapsto \oZ(L_{f})$, as defined in \eqref{eq_star}, is locally determined on $\cfn$. 

For $q\in\R$, Huang, Lutwak, Yang, and Zhang \cite{HLYZ}  introduced the $q$-th dual area measures of a convex body $K$ with the origin in the interior as 
\begin{equation}\label{eq:dam}
    \tilde S_q(K)(B)= \frac1n \int_B \rho_K^q(s) \d\hn(s) 
\end{equation}
for Borel sets $B\subset \sn$. Note that the map $K \mapsto \tilde S_q(K)$ is a valuation defined on convex bodies in $\rn$ with the origin in their interiors and taking values in $\measn$. Also, note that this valuation can be extended to all star bodies using the same definition as in \eqref{eq:dam}. 

Theorem \ref{t:char} is now an immediate consequence of Theorem \ref{t:characterization}. As a consequence of Theorem \ref{t:orlicz}, we also get the following.

\begin{theorem} 
A map $\oZ\colon \starn\to \measn$ is a weak$^{\,*}\!$ continuous, locally determined, rotation equi\-variant valuation with $\oZ(\{0\})=0$ if and only if there is 
a continuous function $\theta\colon\rp\to\rp$ with $\theta(0)=0$ such that
$$\langle \oZ(L),g\rangle =\int_{\sn} \theta(\rho_L(s))g(s)\d\hn( s)$$
for every $L\in\starn$ and $g\in C(\sn)$.
\end{theorem}

\noindent
These measures play a relevant role in the dual Orlicz--Brunn--Minkowski theory (cf.\ \cite{ZhuZhouXu, GardnerHugWeilYe}).

\subsection{A final remark on weak$^*$ continuity versus norm continuity} 
Our purpose here is to exhibit a simple example of a weak$^*$ continuous, rotation equivariant valuation that is not norm continuous, i.e. it is not continuous when $M(\sn)$ is endowed with the norm topology. Take $\rhot_\lambda=\min\{\lambda,1\}\,\delta_{\cos\lambda}$ with  $\delta_\lambda$ the Dirac measure on $\R$. Clearly, $\rhot_0=0$. Let us check that $\lambda\mapsto\rhot_\lambda$ is weak$^*$ continuous. For every $\G\in C([-1,1])$, we have
\[
    \langle \rhot_\lambda,\G\rangle=\min\{\lambda,1\} \G(\cos\lambda)
\]
which is continuous in $\lambda$, since both $\G$ and $\lambda\mapsto\min\{\lambda,1\}$ are continuous. Therefore, by Theorem \ref{th:representation equivariant}, the measure-valued function $\rhot$ defines a weak$^*$ continuous, rotation equivariant valuation $\oV$. Let us show that it is not norm continuous when $n=2$. We will identify $\mathbb{R}^2$ with the complex plane and the unit sphere $S^1$ with the interval $[0,2\pi]$, that is, any angle $\beta\in[0,2\pi]$ is identified with $e^{i\beta}\in S^1$. In this case note that for $g\in C(S^1)$,
\[ \radon g(e^{is}, \cos\lambda)
        =\frac{1}{2}\left(g(e^{i(s+\lambda) })  +g(e^{i(s-\lambda)})\right).
\]
Now consider the function
\[
    f(e^{is})=\left\{\begin{array}{ll}
        -s+2\pi & \text{if }s\in[0,\pi], \\
         s & \text{if }s\in[\pi,2\pi],
    \end{array}  \right.
\]
and for any $m\in\mathbb{N}$ the functions $f_m(e^{is})=f(e^{is})+\frac{1}{m}$ which are continuous. Clearly $f_m$ converges uniformly to $f$ and we have $f,f_m\geq 1$. Thus,
\[
    \langle U'(f(e^{is}),e^{is}),g\rangle
   =\frac{1}{2}\left(g(e^{0})+g(e^{i(2s)})\right)
\]
and 
\begin{align*}
    \langle U'(f_m(e^{is}),e^{is}),g\rangle 
  = \left\{\begin{array}{ll}
       \frac{1}{2}\big(g(e^{i\frac{1}{m}})+g(e^{i\left(2s-\frac{1}{m}\right)}) \big)  &  \text{if }s\in[0,\pi],\\[4pt]
       \frac{1}{2}\big(g(e^{i(2s+\frac{1}{m})})+g(e^{-i\left(\frac{1}{m}\right)}) \big)  & \text{if }s\in[\pi,2\pi],
    \end{array}  \right.
\end{align*}
where we used the fact that $g(e^{i(\beta+2\pi)})=g(e^{i\beta})$ for any $\beta\in\mathbb{R}$. Then, for any $g\in C(S^1)$,
\begin{align*}
    \langle \oV(f),g\rangle &=\frac{1}{2}\int_0^{2\pi}\left(g(e^{0})+g(e^{i2s})\right)\d s
    =\pi g(e^0)+\frac{1}{2}\int_0^{2\pi} g(e^{is})\d s,    
\end{align*}
and
\begin{multline*}
    \langle \oV(f_m),g\rangle =\frac{1}{2}\int_0^{\pi}\left(g(e^{i\frac{1}{m}})+g(e^{i\left(2s-\frac{1}{m}\right)})\right)\d s\\
    +\frac{1}{2}\int_\pi^{2\pi}\left(g(e^{-i\frac{1}{m}})+g(e^{i\left(2s+\frac{1}{m}\right)})\right)\d s \\
     =\frac{\pi}{2} \left(g(e^{i\frac{1}{m}})+g(e^{-i\frac{1}{m}})\right)+\frac{1}{2}\int_0^{2\pi} g(e^{is})\d s.
\end{multline*}
Therefore, for any $m$ one can take $g_m(s)$ such that $g_m(e^0)=1$ and $g_m(e^{\frac{i}{m}})=g(e^{-\frac{i}{m}})=0$ and, we have
\[
    \langle \oV(f)-\oV(f_m),g_m\rangle =\pi.
\]
Thus, $\oV$ is not norm continuous.

\goodbreak
\appendix
\section{Proof of Lemma \ref{l:weak pseudokernel}}\label{A:Inspection}

For the sake of completeness, we include here the details of the proof of Lemma \ref{l:weak pseudokernel}. The approach is the same as in \cite{TrViIntRep}, keeping track of the more general condition of control measure used in this paper. In fact, the same proof will show that one can obtain a similar representation for every measure which satisfies the control condition, given by Definition \ref{d:Mcontrol}, and whose control function is non-decreasing. 

Let $\mu$ and $\zeta$ be the control measure and control function of $\oV$ as defined in Section \ref{S:control}. For every $g\in C(\sn)_{\scriptscriptstyle +}$, every $\lambda\in\mathbb{R}_{\scriptscriptstyle +}$ and every $A\in\Sigma_n$ we have 
\[
    |\nu_{\lambda,g}(A)|\leq \|g\|_{\scriptscriptstyle \infty} \zeta (\lambda)\mu (A).
\]
Also recall that $\mu$ is non-negative, finite, Borel and regular and that $\zeta$ is bounded on bounded sets and non-decreasing. 

For $g\in C(\sn)$, write $\oV_g=\langle \oV,g\rangle$, and denote by $\overline \oV_g$ its continuous extension to $B(\sn)_{\scriptscriptstyle +}$ as in \cite{TrViIntRep}.

\begin{proposition}\label{p:appendix}
Let $g\in\cfn$. There is a continuous function $\Phi_g:B(\sn)_{\scriptscriptstyle +}\to L_1(\mu)$ such that 
\[
    \overline\oV_g(f)=\int_{\sn}\Phi_g(f)\d \mu.
\]
Moreover, $\Phi_g(f\chi_A)=\Phi_g(f)\chi_A$ for every $A\in\Sigma_n$.
\end{proposition}

\begin{proof}
Fix $f\in B(\sn)_{\scriptscriptstyle +}$. For every $A\in\Sigma_n$ and every $g\in\cfn$ define
\[
    \nu_{f,g}(A)=\overline\oV_g(f\chi_A).
\] 
Since $\overline\oV_g$ is a valuation, for every $A_1,A_2\in\Sigma_n$ with $A_1\cap A_2=\emptyset$, we have
\[
    \nu_{f,g}(A_1\cup A_2)=\nu_{f,g}(A_1)+\nu_{f,g}(A_2). 
\]
Therefore, $\nu_{f,g}$ is finitely additive. Let us now fix $A\in \Sigma_n$ and let $h_m=\sum_{i=1}^{i_m}\lambda_{i,m} \chi_{A_{i,m}}$ be a sequence of simple functions that converge to $f\chi_A$ in the uniform norm with $0\leq h_m\leq  f\chi_A$ and $A_{i,m}\cap A_{j,m}=\emptyset$ whenever $j\neq i$. Then, by continuity of $\overline\oV_g$ 
\[
    \nu_{f,g}(A)=\overline\oV_g(f\chi_A)=\lim_{m\to\infty}\overline\oV_g(h_m)=\lim_{m\to\infty}\sum_{i=1}^{i_m}\nu_{\lambda_{i,m},g}(A_{i,m})
\]
with $A_{i,m}\subset A$. Therefore,
\begin{align*}
    |\nu_{f,g}(A)|&\leq \lim_{m\to\infty}\sum_{i=1}^{i_m}|\nu_{\lambda_{i,m},g}(A_{i,m})|\leq \lim_{m\to\infty}\sum_{i=1}^{i_m}\|g\|_{\scriptscriptstyle \infty} \zeta(\lambda_{i,m})\mu(A_{i,m})\\
    & \leq \|g\|_{\scriptscriptstyle \infty} \zeta(\|f\|_{\scriptscriptstyle \infty}) \lim_{m\to\infty}\mu\left(\bigcup_{i}A_{i,m}\right)\leq \|g\|_{\scriptscriptstyle \infty} \zeta(\|f\|_{\scriptscriptstyle \infty}) \mu(A),
\end{align*}
since $\mu$ is non-negative and countably additive. From this bound, we get that $\nu_{f,g}$ is countably additive and absolutely continuous with respect to $\mu$.

Let $\Phi_g(f)$ be the Radon--Nikodym derivative of $\nu_{f,g}$ with respect to $\mu$. Since $g\in\cfn$, we have $\nu_{\lambda,g}\geq 0$ and $\Phi_g(f)\geq 0$ for every $f\in B(\sn)_{\scriptscriptstyle +}$. For every $A\in\Sigma_n$, we have
\[
    \overline\oV_g(f\chi_A)=\nu_{f,g}(A)=\int_A \Phi_g(f)\d \mu.
\]
Moreover, for every $A,B\in\Sigma_n$,
\[
    \int_{A} \Phi_g(f)\chi_B\d \mu=\int_{A\cap B} \Phi_g(f)\d \mu=\nu(f\chi_A\chi_B)=\int_{A} \Phi_g(f\chi_B)\d \mu,
\]
and it follows that $\Phi_g(f\chi_A)=\Phi_g(f)\chi_A$ as elements in $L^1(\mu)$. 

Now, we need to check that $\Phi_g$ is continuous. First, let $f\in B(\sn)_{\scriptscriptstyle +}$ and $\varepsilon>0$. By the continuity of $\overline\oV_g$, there exists $\delta>0$ such that if $h\in B(\sn)_{\scriptscriptstyle +}$ with $\|f-h\|<\delta$ then $|\overline\oV_g(f)-\overline\oV_g(h)|<\varepsilon$. Now, let $A\in \Sigma_n$ and $f'\in B(\sn)_{\scriptscriptstyle +}$ with $\|f-f'\|<\delta$. Let $h=f'\chi_A+f\chi_{A^c}$, which clearly satisfies $\|f-h\|<\delta$. It follows that
\[
    \Big|\int_A \Phi_g(f)-\Phi_g(f')\d \mu\Big|=|\overline\oV_g(f\chi_A)-\overline\oV_g(f'\chi_A)|=|\overline\oV_g(f)-\overline{V_g}(h)|<\varepsilon.
\]
Now, given $f\in B(\sn)_{\scriptscriptstyle +}$ and $\varepsilon>0$, let $\delta>0$ be as before. Suppose $f'\in B(\sn)_{\scriptscriptstyle +}$ satisfies $\|f-f'\|<\delta$. Let us consider
\[
    A=\{t\in\sn:\Phi_g(f)(t)-\Phi_g(f')(t)>0\}.
\]
We clearly have
\[
    \|\Phi_g(f)-\Phi_g(f')\|=\int_A \Phi_g(f)-\Phi_g(f')\d \mu+\int_{A^c} \Phi_g(f)-\Phi_g(f')\d \mu\leq 2\varepsilon.
\]
\end{proof}

For each $g\in C(\sn)_+$ and $\lambda\in\mathbb{R}_+$, choose a Borel and non-negative representative of $\Phi_g(\lambda{\uno})$, which we denote by
\begin{equation}\label{eq:def_K0}
    K_0(\lambda,s)=\Phi_g(\lambda{\uno})(s).
\end{equation}
We may also assume that $|K_0(\lambda,s)|\leq \zeta (\lambda)\|g\|_{\scriptscriptstyle \infty}$, by the control property.

For $\delta>0$, $\lambda>0$, and $A\in\Sigma_n$, let 
\[
    \omega_\lambda(\delta,A)=\sup\Big\{\int_A|K_0(\alpha,s)-K_0(\alpha',s')|\d \mu(s):\alpha,\alpha'\in[0,\lambda],|\alpha-\alpha'|\leq \delta  \Big\},
\]
and we define
\[
    \omega_\lambda(\delta)=\sup\big\{\sum_{i=1}^m\omega_\lambda(\delta,A_i):\bigcup_{i=1}^mA_i=\sn, A_i\cap A_j=\emptyset \text{ for }i\neq j\big\}.
\]
First, note that the previous object is finite. Indeed, since $K_0(\lambda,s)\geq 0$,
\begin{align*}
    \int_A|K_0(\alpha,s)-K_0(\alpha',s')|\d \mu(s)&\leq\nu_{g,\alpha}(A)+\nu_{g,\alpha'}(A)\\
    &\leq \|g\|_{\scriptscriptstyle \infty} \mu(A) (\zeta(\alpha)+\zeta(\alpha')).
\end{align*}
Hence, since $\zeta$ is bounded on bounded sets,
\[
    \omega_\lambda(\delta)\leq 2\|g\|_{\scriptscriptstyle \infty} \mu(\sn)\sup_{\alpha\leq \lambda}\zeta(\alpha)<\infty.
\]

\begin{lemma}\label{l:limit_omega}
For every $\lambda>0$, we have 
\[
    \lim_{\delta\to0}\omega_\lambda(\delta)=0.
\]
\end{lemma}
\begin{proof}
Given $\varepsilon>0$, since $\oV_g$ is uniformly continuous on bounded sets \cite[Theorem 3.8]{TrViIntRep}, there is $\delta>0$ such that $|\overline\oV_g(f)-\overline\oV_g(f')|\leq \frac{\varepsilon}{3}$ whenever $f,f'\in B(\sn)_{\scriptscriptstyle +}$ satisfy $\|f\|_{\scriptscriptstyle \infty},\|f'\|_{\scriptscriptstyle \infty}\leq \lambda$ and $\|f-f'\|_{\scriptscriptstyle \infty}<\delta$.

Let $(A_i)_{i=1}^m\subset\Sigma_n$ pairwise disjoint with $\bigcup_{i=1}^mA_i=\sn$ such that
\[
    \omega_\lambda(\delta)\leq \sum_{i=1}^m\omega_\lambda(\delta,A_i)+\frac{\varepsilon}{3}.
\]
For $1\leq i\leq m$, let $\alpha_i,\alpha_i'\in[0,\lambda]$ with $|\alpha_i-\alpha_i'|\leq \delta$, such that
\[
    \omega_\lambda(\delta,A_i)\leq \int_{A_i}|K_0(\alpha_i,s)-K_0(\alpha_i',s)|\d \mu(s)+\frac{\varepsilon}{3m}.
\]
Let $A_i^+=\{s\in A_i:K_0(\alpha_i,s)\geq K_0(\alpha_i',s)\}$ and $A_i^-=\{s\in A_i: K_0(\alpha_i,s)<K_0(\alpha_i',s)\}$, which belong to $\Sigma_n$. Set
\[
    f=\sum_{i=1}^m \alpha_i\chi_{A_i^+}+\alpha_i'\chi_{A_i^-}\quad \text{and}\quad f'=\sum_{i=1}^m \alpha_i\chi_{A_i^-}+\alpha_i'\chi_{A_i^+}.
\]
Clearly, we have $\|f\|_{\scriptscriptstyle \infty},\|f'\|_{\scriptscriptstyle \infty}\leq\lambda$ and $\|f-f'\|_{\scriptscriptstyle \infty}\leq \delta$. Hence, $|\overline\oV_g(f)-\overline\oV_g(f')|\leq \frac{\varepsilon}{3}$, which yields
\begin{align*}
    \omega_\lambda(\delta) &\leq \sum_{i=1}^m\int_{A_i}|K_0(\alpha_i,s)-K_0(\alpha_i',s)|\d \mu(s)   +\frac{2\varepsilon}{3}\\
    &=\sum_{i=1}^m\int_{A_i^+}K_0(\alpha_i,s)-K_0(\alpha_i',s)\d \mu(s) \\
    & \quad\hspace{25mm} +\sum_{i=1}^m\int_{A_i^-}K_0(\alpha_i',s)-K_0(\alpha_i,s)\d \mu(s)  +\frac{2\varepsilon}{3}\\
    &=\sum_{i=1}^m\overline\oV_g(\alpha_i\chi_{A_i^+})-\overline\oV_g(\alpha_i'\chi_{A_i^+})+\overline\oV_g(\alpha_i'\chi_{A_i^-})-\overline\oV_g(\alpha_i\chi_{A_i^-})+\frac{2\varepsilon}{3} \\
    &=\overline\oV_g(f)-\overline\oV_g(f')+\frac{2\varepsilon}{3}\leq \varepsilon.
\end{align*}
\end{proof}

\begin{lemma}\label{l:appendix}
There exists a Borel set $A_0\subset\sn$ with $\mu(A_0)=0$ such that $K_0(\cdot,s)$ is uniformly continuous on every bounded set of rational numbers  for every $s\not\in A_0$.
\end{lemma}
\begin{proof}
For $k\in\mathbb{N}$, let $Q_k=\mathbb{Q}\cap [0,k]$. Given $\delta>0$ and $\varepsilon>0$, let
\begin{multline}
    A(\delta,\varepsilon)\\
    =\left\{s\in\sn\colon \sup\left\{|K_0(\lambda,s)-K_0(\lambda',s)|\colon \lambda,\lambda'\in Q_k,|\lambda-\lambda'|\leq\delta  \right\}>\varepsilon   \right\}.
\end{multline}
Given $\lambda,\lambda'\in Q_k$ with $|\lambda-\lambda'|\leq \delta$, define
\[
    B(\lambda,\lambda',\delta,\varepsilon)=\{s\in\sn:|K_0(\lambda,s)-K_0(\lambda',s)|>\varepsilon\}.
\]
Let $(\lambda_i,\lambda_i')_{i\in\mathbb{N}}$ be an enumeration of all pairs $(\lambda,\lambda')$ where $\lambda,\lambda'\in Q_k$ and $|\lambda-\lambda'|\leq\delta$. Let $A_1(\delta,\varepsilon)=B(\lambda_1,\lambda_1',\delta,\varepsilon)$ and
\[
    A_i(\delta,\varepsilon)=B(\lambda_i,\lambda_i',\delta,\varepsilon)\setminus\bigcup_{j=1}^{i-1}A_j(\delta,\varepsilon).
\]
Clearly, $\bigcup_{i=1}^\infty A_i(\delta,\varepsilon)=A(\delta,\varepsilon)$. Obviously, $B(\lambda_i,\lambda_i',\delta,\varepsilon)$ is Borel and, therefore, $A(\delta,\varepsilon)$ is Borel as well. It follows that
\begin{align*}
    \varepsilon\mu\left(A(\delta,\varepsilon)\right)&= \varepsilon\sum_{i=1}^\infty\mu(A_i(\delta,\varepsilon))=\sum_{i=1}^\infty \int_{A_i(\delta,\varepsilon)}\varepsilon\d \mu\\
    &\leq \sum_{i=1}^\infty \int_{A_i(\delta,\varepsilon)}|K_0(\lambda_i,s)-K_0(\lambda_i',s)|\d \mu(s)\leq \omega_k(\delta).
\end{align*}
By Lemma \ref{l:limit_omega}, for every $\varepsilon>0$ we have 
\[
    \lim_{\delta\to0}\mu\left(A(\delta,\varepsilon)\right)=0.
\]
Now, for each $\varepsilon>0$, pick a sequence $\delta_m\to 0$ such that
\[
    \sum_{i=1}^\infty \mu( A(\delta_m,\varepsilon))<\infty,
\]
and set
\[
    A(\varepsilon)=\bigcap_{k=1}^\infty\bigcup_{m=k}^\infty A(\delta_m,\varepsilon).
\]
Clearly, $A(\varepsilon)$ is Borel and
\[
    \mu(A(\varepsilon))=\lim_{k\to\infty}\mu\left(\bigcup_{m=k}^\infty A(\delta_m,\varepsilon)\right)\leq \lim_{k\to\infty}\sum_{m=k}^\infty\mu(A(\delta_m,\varepsilon))=0.
\]
Now, take $\varepsilon_j\to 0$ and set $A^k=\bigcup_{j=1}^\infty A(\varepsilon_j)$ which is also Borel and $\mu(A^k)=0$. Thus, $K_0(\cdot,s)$ is uniformly continuous on $Q_k$ for every $s\in \sn\setminus A^k$.

As a final step, set $A_0=\bigcup_{k=1}^\infty A^k$, which is Borel with $\mu(A_0)=0$. Hence,  $K_0(\cdot,s)$
is uniformly continuous on every bounded set of rational numbers for every $s\not\in A_0$.
\end{proof}

\begin{lemma}
For $g\in\cfn$, there is $K_g\colon \rp\times \sn \to \R$ such that
$$\oV_g(f)= \int_{\sn } K_g(f(s), s) \d\mu(s)$$
for every $f\in \bfn$. Moreover, $K_g$ satisfies the following conditions.
\begin{itemize}
    \item For each $\lambda\ge 0$, the function $K_g(\lambda,\cdot)$ coincides with the Radon--Nikodym derivative with respect to $\mu$ of the measure $\nu_{\lambda,g}$ associated to the scalar-valued valuation $\oV_g$ according to \eqref{eq:def nu lambda}, and in particular, we have $K_g(\lambda, \cdot)\in L_1(\sn,\mu)$.
    \item The function $K_g(\cdot,s)$ is continuous for every $s\in \sn $.
\end{itemize}
\end{lemma}
\begin{proof}
Let $K_0(\lambda,s)$ be defined as in \eqref{eq:def_K0} and recall that $0\leq K_0(\lambda,s)\leq \zeta(\lambda)\|g\|_{\scriptscriptstyle \infty} $ for any $\lambda\in\rp$ and any $s\in\sn$. By Lemma \ref{l:appendix}, we can define
\[
    K_g(\lambda,s)=\left\{ \begin{array}{cl}
        \displaystyle\lim_{m}K_0(\lambda_m,s) & \text{if }s\not\in A_0,\text{ and }\lambda_m\to \lambda\text{ with }\lambda_m\in\mathbb{Q} \\
        0 & \text{if }s\in A_0
    \end{array}   \right.
\]
Clearly, for every fixed $s\in\sn$ the mapping $\lambda\mapsto K_g(\lambda,s)$ is continuous and for every $\lambda\geq0$. If $\lambda\in\mathbb{Q}_{\scriptscriptstyle +}$, then the mapping $s\mapsto K_g(\lambda,s)=K_0(\lambda, s)\chi_{A_0^c}(s)$ is Borel, since $s\mapsto K_0(\lambda,s)$ is Borel and $s\mapsto \chi_{A_0^c}$ is also. If $\lambda\not\in\mathbb{Q}_{\scriptscriptstyle +}$, the mapping $s\mapsto \lim_{m}K_0(\lambda_m,s)$ is Borel, since it is the limit of Borel mappings and, thus, $s\mapsto K_g(\lambda,s)$ is also Borel.

We check that $s\mapsto K_g(\lambda,s)$ is the Radon--Nikodym derivative with respect to $\mu$ of the measure $\nu_{\lambda,g}$. If $\lambda\in\mathbb{Q}_{\scriptscriptstyle +}$, $K_0(\lambda,s)=\frac{d\nu_{\lambda,g}}{d\mu}$ and since $\mu(A_0)=0$ one has that $K_0(\lambda,\cdot)=K_g(\lambda,\cdot)$ as elements of $L^1(\mu)$.

If $\lambda\not\in\mathbb{Q}_n$, take $\{\lambda_m\}_{m\in\mathbb{N}}\subset\mathbb{Q}_{\scriptscriptstyle +}$ such that $\lambda_m\to \lambda$. Since $\overline\oV_g$ is continuous,
\[
    \nu_{\lambda,g}(A)=\overline\oV_g(\lambda\chi_{A})=\lim_{m}\overline\oV_g(\lambda_m\chi_A)=\lim_m\int_A K_g(\lambda_m,s)\d \mu(s).
\]
Now, $K_g(\lambda_m,s)\leq K_0(\lambda_m,s)\leq \|g\|_{\scriptscriptstyle \infty} \zeta(\lambda_m)$. As $\zeta$ is bounded on bounded sets, the Dominated Convergence Theorem yields
\[
    \nu_{\lambda,g}(A)=\int_A K_g(\lambda,s)\d \mu(s).
\]
Since this holds for every $A\in\Sigma_n$, $K_g(\lambda,\cdot)$ coincides with $\frac{d\nu_{\lambda,g}}{d\mu}$ as elements in $L^1(\mu)$.

As a final step, we prove the representation. In Proposition \ref{p:appendix}, we checked that $\Phi_g(f\chi_A)=\Phi_g(f)\chi_A$, that is,
\[
    \int_\sn K_0(\lambda\chi_A(s),s)\d \mu(s)=\int_A K_0(\lambda,s)\d \mu(s)=\nu_{\lambda,g}(A)=\overline\oV_g(\lambda\chi_A).
\]
Hence,
\[
    \overline\oV_g(\lambda\chi_A)=\int_\sn K_g(\lambda\chi_A(s),s)\d \mu(s)
\]
holds for indicator functions with rational coefficients. It is easy to check that it can be extended to simple functions, and these are in turn dense in the set of bounded Borel functions.

Now, to extend the representation for every bounded Borel function we will show that the mapping on the right side of the equation, that is,
\[
    f\mapsto \int_\sn K_g(f(s),s)\d \mu(s),
\]
is a continuous map. Therefore, since it coincides with $\overline\oV_g$ on a dense set, they will coincide on the bounded Borel functions. Let $f\in B(\sn)_{\scriptscriptstyle +}$ and $\{f_m\}_{m\in\mathbb{N}}\subset B(\sn)_{\scriptscriptstyle +}$ such that $\|f-f_m\|_{\scriptscriptstyle \infty}\to 0$ as $m\to\infty$. In particular, $f_m(s)\to f(s)$ as $m\to\infty$ for every $s\in\sn$, which implies that $K_g(f_m(s),s)\to K_g(f(s),s)$ as $m\to\infty$. For every $\lambda\geq 0$ and every $s\in\sn$ we have
\[
    K_g(\lambda,s)\leq \|g\|_{\scriptscriptstyle \infty} \zeta(\lambda).
\]
Hence, taking $\lambda=\sup_{m}\|f_m\|$ and using the fact that $\zeta$ is bounded on bounded sets we get
\[
    K_g(f_m(s),s)\leq \|g\|_{\scriptscriptstyle \infty}  \sup_{\alpha\leq \lambda}\zeta(\alpha).
\]
Applying the Dominated Convergence Theorem it follows that
\[
    \lim_{m\to\infty}\int_\sn K_g(f_m(s),s)\d \mu(s)=\int_\sn K_g(f(s),s)\d \mu(s),
\]
which concludes the proof.
\end{proof}
Finally, note that if $g\in C(\sn)$, then
\[
    \oV_g=\oV_{g^+}-\oV_{g^-}.
\]
Therefore, the function defined as
\[
    K_{g}(\lambda,s)=K_{g^+}(\lambda,s)-K_{g^-}(\lambda,s)
\]
will satisfy Lemma \ref{l:weak pseudokernel}.

\begin{remark}
The previous result can also be obtained by an application of the Radon--Nikodym theorem: if a scalar valuation $\oW$ admits an integral representation with respect to a measure $\sigma$, then for every measure $\eta$ such that $\sigma$ is absolutely continuous with respect to $\eta$, the valuation $\oW$ admits a similar integral representation with respect to $\eta$. Finally, one can check that the control measure defined in \cite{TrViIntRep} is absolutely continuous with respect to the control measure $\mu$ defined in Section \ref{S:control}, which yields the result.
\end{remark}

\bigskip
\bibliography{measurevalued}

@article {Alesker,
    AUTHOR = {Alesker, S.},
     TITLE = {Valuations on convex functions and convex sets and
              {M}onge--{A}mp\`ere operators},
   JOURNAL = {Adv. Geom.},
  FJOURNAL = {Advances in Geometry},
    VOLUME = {19},
      YEAR = {2019},
      PAGES = {313--322},
      ISSN = {1615-715X,1615-7168},
   MRCLASS = {52B45 (52A20 52A41)},
  MRNUMBER = {3982569},
MRREVIEWER = {Fabian\ Mussnig},
       DOI = {10.1515/advgeom-2018-0031},
       URL = {https://doi.org/10.1515/advgeom-2018-0031},
}

@book {C,
    AUTHOR = {Cohn, D. L.},
     TITLE = {Measure Theory},
    SERIES = {Birkh\"{a}user Advanced Texts: Basler Lehrb\"{u}cher.},
   EDITION = {Second},
 PUBLISHER = {Birkh\"{a}user/Springer, New York},
      YEAR = {2013},
     PAGES = {xxi+457},
      ISBN = {978-1-4614-6955-1; 978-1-4614-6956-8},
   MRCLASS = {28-01},
  MRNUMBER = {3098996},
MRREVIEWER = {Ville Suomala},
       DOI = {10.1007/978-1-4614-6956-8},
       URL = {https://doi.org/10.1007/978-1-4614-6956-8},
}

@article {CLM2017,
    AUTHOR = {Colesanti, A. and Ludwig, M. and Mussnig, F.},
     TITLE = {Minkowski valuations on convex functions},
   JOURNAL = {Calc. Var. Partial Differential Equations},
  FJOURNAL = {Calculus of Variations and Partial Differential Equations},
    VOLUME = {56},
      YEAR = {2017},
     PAGES = {Paper No. 162},
      ISSN = {0944-2669,1432-0835},
   MRCLASS = {52B45 (26B25 46B20 46E35 52A21 52A41)},
  MRNUMBER = {3715395},
MRREVIEWER = {Thomas\ Jahn},
       DOI = {10.1007/s00526-017-1243-4},
       URL = {https://doi.org/10.1007/s00526-017-1243-4},
}

@article {CLM2019,
    AUTHOR = {Colesanti, A. and Ludwig, M. and Mussnig, F.},
     TITLE = {Valuations on convex functions},
   JOURNAL = {Int. Math. Res. Not. IMRN},
  FJOURNAL = {International Mathematics Research Notices. IMRN},
      YEAR = {2019},
     PAGES = {2384--2410},
      ISSN = {1073-7928,1687-0247},
   MRCLASS = {52A41 (52B45)},
  MRNUMBER = {3942165},
MRREVIEWER = {Thomas\ Jahn},
       DOI = {10.1093/imrn/rnx189},
       URL = {https://doi.org/10.1093/imrn/rnx189},
}

@article {CLM2020,
    AUTHOR = {Colesanti, A. and Ludwig, M. and Mussnig, F.},
     TITLE = {A homogeneous decomposition theorem for valuations on convex
              functions},
   JOURNAL = {J. Funct. Anal.},
  FJOURNAL = {Journal of Functional Analysis},
    VOLUME = {279},
      YEAR = {2020},
     PAGES = {108573},
      ISSN = {0022-1236,1096-0783},
   MRCLASS = {52B45 (26B25 52A21 52A41)},
  MRNUMBER = {4097279},
MRREVIEWER = {N.\ Lombardi},
       DOI = {10.1016/j.jfa.2020.108573},
       URL = {https://doi.org/10.1016/j.jfa.2020.108573},
}

@article {CLM2023,
    AUTHOR = {Colesanti, A. and Ludwig, M. and Mussnig, F.},
     TITLE = {The {H}adwiger theorem on convex functions, {IV}: {T}he
              {K}lain approach},
   JOURNAL = {Adv. Math.},
  FJOURNAL = {Advances in Mathematics},
    VOLUME = {413},
      YEAR = {2023},
     PAGES = {Paper No. 108832},
      ISSN = {0001-8708,1090-2082},
   MRCLASS = {52B45 (26B25 49Q20 52A39 52A41)},
  MRNUMBER = {4526493},
MRREVIEWER = {Aris\ Daniilidis},
       DOI = {10.1016/j.aim.2022.108832},
       URL = {https://doi.org/10.1016/j.aim.2022.108832},
}

@article {CLM2022,
    AUTHOR = {Colesanti, A. and Ludwig, M. and Mussnig, F.},
     TITLE = {The {H}adwiger theorem on convex functions, {III}: {S}teiner
              formulas and mixed {M}onge-{A}mp\`ere measures},
   JOURNAL = {Calc. Var. Partial Differential Equations},
  FJOURNAL = {Calculus of Variations and Partial Differential Equations},
    VOLUME = {61},
      YEAR = {2022},
     PAGES = {Paper No. 181},
      ISSN = {0944-2669,1432-0835},
   MRCLASS = {52A41 (26B25 49Q20 52A39 52B45)},
  MRNUMBER = {4453228},
}

@Article{CLM2024b,
  author  = {Colesanti, A. and Ludwig, M. and Mussnig, F.},
  journal = {Amer. J. Math.},
  title   = {The {H}adwiger theorem on convex functions,~{II}. {C}auchy--{K}ubota formulas},
  year ={2024},
  volume ={147},
 pages= {927-955},}

@Article{CLM2024,
  AUTHOR = {Colesanti, A. and Ludwig, M. and Mussnig, F.},
     TITLE = {The {H}adwiger theorem on convex functions,~{I}},
   JOURNAL = {Geom. Funct. Anal.},
  FJOURNAL = {Geometric and Functional Analysis},
    VOLUME = {34},
      YEAR = {2024},
       PAGES = {1839--1898},
      ISSN = {1016-443X,1420-8970},
   MRCLASS = {52A41 (26B25 49Q20 52A39 52B45)},
  MRNUMBER = {4823212},
       DOI = {10.1007/s00039-024-00693-8},
       URL = {https://doi.org/10.1007/s00039-024-00693-8},
}

@article {CPTV2020,
    AUTHOR = {Colesanti, A. and Pagnini, D. and Tradacete, P.
              and Villanueva, I.},
     TITLE = {A class of invariant valuations on {${\rm Lip}(S^{n-1})$}},
   JOURNAL = {Adv. Math.},
  FJOURNAL = {Advances in Mathematics},
    VOLUME = {366},
      YEAR = {2020},
     PAGES = {107069},
      ISSN = {0001-8708,1090-2082},
   MRCLASS = {52A20 (26A16 28A78)},
       DOI = {10.1016/j.aim.2020.107069},
       URL = {https://doi.org/10.1016/j.aim.2020.107069},
}

@article {CPTV2021,
    AUTHOR = {Colesanti, A. and Pagnini, D. and Tradacete, P.
              and Villanueva, I.},
     TITLE = {Continuous valuations on the space of {L}ipschitz functions on
              the sphere},
   JOURNAL = {J. Funct. Anal.},
  FJOURNAL = {Journal of Functional Analysis},
    VOLUME = {280},
      YEAR = {2021},
     PAGES = {Paper No. 108873},
      ISSN = {0022-1236,1096-0783},
   MRCLASS = {52B45 (26A16 28A12)},
       DOI = {10.1016/j.jfa.2020.108873},
       URL = {https://doi.org/10.1016/j.jfa.2020.108873},
}

@article {Diestel-Faires,
    AUTHOR = {Diestel, J. and Faires, B.},
     TITLE = {On vector measures},
   JOURNAL = {Trans. Amer. Math. Soc.},
  FJOURNAL = {Transactions of the American Mathematical Society},
    VOLUME = {198},
      YEAR = {1974},
     PAGES = {253--271},
      ISSN = {0002-9947},
   MRCLASS = {46G10 (28A45)},
       DOI = {10.2307/1996758},
       URL = {https://doi.org/10.2307/1996758},
}

@book {fabianetal,
    AUTHOR = {Fabian, M. and Habala, P. and H\'{a}jek, P. and Montesinos
              Santaluc\'{\i}a, V. and Pelant, J. and Zizler, V.},
     TITLE = {Functional Analysis and Infinite-dimensional Geometry},
    SERIES = {CMS Books in Mathematics/Ouvrages de Math\'{e}matiques de la SMC},
    VOLUME = {8},
 PUBLISHER = {Springer-Verlag, New York},
      YEAR = {2001},
     PAGES = {x+451},
      ISBN = {0-387-95219-5},
   MRCLASS = {46-01 (46B20 46Bxx)},
       DOI = {10.1007/978-1-4757-3480-5},
       URL = {https://doi.org/10.1007/978-1-4757-3480-5},
}

@book {Federer,
    AUTHOR = {Federer, H.},
     TITLE = {Geometric Measure Theory},
    SERIES = {Die Grundlehren der mathematischen Wissenschaften, Band 153},
 PUBLISHER = {Springer-Verlag New York, Inc., New York},
      YEAR = {1969},
     PAGES = {xiv+676},
   MRCLASS = {28.80 (26.00)},
}

@book {Gardner,
    AUTHOR = {Gardner, R.},
     TITLE = {Geometric {T}omography},
    SERIES = {Encyclopedia of Mathematics and its Applications},
    VOLUME = {58},
   EDITION = {Second},
 PUBLISHER = {Cambridge University Press, New York},
      YEAR = {2006},
     PAGES = {xxii+492},
      ISBN = {0-521; 0-521-68493-5},
   MRCLASS = {52A22 (44A12 92C55)},
       DOI = {10.1017/CBO9781107341029},
       URL = {https://doi.org/10.1017/CBO9781107341029},
}

@article {GardnerHugWeilYe,
    AUTHOR = {Gardner, R. and Hug, D. and Weil, W. and Ye,
              D.},
     TITLE = {The dual {O}rlicz--{B}runn--{M}inkowski theory},
   JOURNAL = {J. Math. Anal. Appl.},
  FJOURNAL = {Journal of Mathematical Analysis and Applications},
    VOLUME = {430},
      YEAR = {2015},
     PAGES = {810--829},
      ISSN = {0022-247X,1096-0813},
   MRCLASS = {52A39 (52A40)},
  MRNUMBER = {3351982},
MRREVIEWER = {Qingzhong\ Huang},
       DOI = {10.1016/j.jmaa.2015.05.016},
       URL = {https://doi.org/10.1016/j.jmaa.2015.05.016},
}

@book {Hadwiger,
    AUTHOR = {Hadwiger, H.},
     TITLE = {Vorlesungen \"uber {I}nhalt, {O}berfl\"ache und
              {I}soperimetrie},
 PUBLISHER = {Springer-Verlag, Berlin-G\"ottingen-Heidelberg},
      YEAR = {1957},
     PAGES = {xiii+312},
   MRCLASS = {52.00 (28.00)},
}

@article {HLYZ,
    AUTHOR = {Huang, Y. and Lutwak, E. and Yang, D. and Zhang,G.},
     TITLE = {Geometric measures in the dual {B}runn-{M}inkowski theory and
              their associated {M}inkowski problems},
   JOURNAL = {Acta Math.},
  FJOURNAL = {Acta Mathematica},
    VOLUME = {216},
      YEAR = {2016},
     PAGES = {325--388},
      ISSN = {0001-5962},
   MRCLASS = {52A38 (35J20 35J96)},
       DOI = {10.1007/s11511-016-0140-6},
       URL = {https://doi.org/10.1007/s11511-016-0140-6},
}

@article {Klain1996,
    AUTHOR = {Klain, D.},
     TITLE = {Star valuations and dual mixed volumes},
   JOURNAL = {Adv. Math.},
  FJOURNAL = {Advances in Mathematics},
    VOLUME = {121},
      YEAR = {1996},
     PAGES = {80--101},
      ISSN = {0001-8708,1090-2082},
   MRCLASS = {52A39 (52B45)},
       DOI = {10.1006/aima.1996.0048},
       URL = {https://doi.org/10.1006/aima.1996.0048},
}

@article {Klain1997,
    AUTHOR = {Klain, D.},
     TITLE = {Invariant valuations on star-shaped sets},
   JOURNAL = {Adv. Math.},
  FJOURNAL = {Advances in Mathematics},
    VOLUME = {125},
      YEAR = {1997},
     PAGES = {95--113},
      ISSN = {0001-8708,1090-2082},
   MRCLASS = {52A30 (52B45)},
       DOI = {10.1006/aima.1997.1601},
       URL = {https://doi.org/10.1006/aima.1997.1601},
}

@article {Knoerr2024,
    AUTHOR = {Knoerr, J.},
     TITLE = {Smooth valuations on convex functions},
   JOURNAL = {J. Differential Geom.},
  FJOURNAL = {Journal of Differential Geometry},
    VOLUME = {126},
      YEAR = {2024},
     PAGES = {801--835},
      ISSN = {0022-040X,1945-743X},
   MRCLASS = {52A41 (26B25 52B45 53C65)},
       DOI = {10.4310/jdg/1712344223},
       URL = {https://doi.org/10.4310/jdg/1712344223},
}

@article {Knoerr2024b,
    AUTHOR = {Knoerr, J.},
     TITLE = {Monge-{A}mp\`ere operators and valuations},
   JOURNAL = {Calc. Var. Partial Differential Equations},
  FJOURNAL = {Calculus of Variations and Partial Differential Equations},
    VOLUME = {63},
      YEAR = {2024},
     PAGES = {Paper No. 89},
      ISSN = {0944-2669,1432-0835},
   MRCLASS = {52B45 (26B25 35J96 52A39 53C65)},
       DOI = {10.1007/s00526-024-02698-5},
       URL = {https://doi.org/10.1007/s00526-024-02698-5},
}

@article {KnoerrUlivelli,
    AUTHOR = {Knoerr, J. and Ulivelli, J.},
     TITLE = {From valuations on convex bodies to convex functions},
   JOURNAL = {Math. Ann.},
  FJOURNAL = {Mathematische Annalen},
    VOLUME = {390},
      YEAR = {2024},
      PAGES = {5987--6011},
      ISSN = {0025-5831,1432-1807},
   MRCLASS = {52B45 (26B25 53C65)},
  MRNUMBER = {4816127},
MRREVIEWER = {Dan\ Ma},
       DOI = {10.1007/s00208-024-02902-z},
       URL = {https://doi.org/10.1007/s00208-024-02902-z},
}

@article {LiL,
    AUTHOR = {Li, J. and Ludwig, M.},
     TITLE = {Monge--{A}mp\`ere measures and valuations},
   JOURNAL = {Preprint},
      YEAR = {2026},
}

@article {ML2011,
    AUTHOR = {Ludwig, M.},
     TITLE = {Fisher information and matrix-valued valuations},
   JOURNAL = {Adv. Math.},
  FJOURNAL = {Advances in Mathematics},
    VOLUME = {226},
      YEAR = {2011},
     PAGES = {2700--2711},
      ISSN = {0001-8708,1090-2082},
   MRCLASS = {46E35 (52A20 52B45 62B10 94A17)},
       DOI = {10.1016/j.aim.2010.08.021},
       URL = {https://doi.org/10.1016/j.aim.2010.08.021},
}

@article {ML2012,
    AUTHOR = {Ludwig, M.},
     TITLE = {Valuations on {S}obolev spaces},
   JOURNAL = {Amer. J. Math.},
  FJOURNAL = {American Journal of Mathematics},
    VOLUME = {134},
      YEAR = {2012},
     PAGES = {827--842},
      ISSN = {0002-9327,1080-6377},
   MRCLASS = {52B45 (52A20 52A39 52B11)},
       DOI = {10.1353/ajm.2012.0019},
       URL = {https://doi.org/10.1353/ajm.2012.0019},
}

@incollection {ML2023,
    AUTHOR = {Ludwig, M.},
     TITLE = {Geometric valuation theory},
 BOOKTITLE = {European {C}ongress of {M}athematics},
     PAGES = {93--123},
 PUBLISHER = {EMS~Press, Berlin},
      YEAR = {2023},
      ISBN = {978-3-98547-051-8; 978-3-98547-551-3},
   MRCLASS = {52B45 (26B25 46E35 52A20)},
}

@article {Lutwak75,
    AUTHOR = {Lutwak, E.},
     TITLE = {Dual mixed volumes},
   JOURNAL = {Pacific J. Math.},
  FJOURNAL = {Pacific Journal of Mathematics},
    VOLUME = {58},
      YEAR = {1975},
     PAGES = {531--538},
      ISSN = {0030-8730,1945-5844},
   MRCLASS = {52A40},
       URL = {http://projecteuclid.org/euclid.pjm/1102905685},
}

@article {LiMA,
    AUTHOR = {Li, J. and Ma, D.},
     TITLE = {Laplace transforms and valuations},
   JOURNAL = {J. Funct. Anal.},
  FJOURNAL = {Journal of Functional Analysis},
    VOLUME = {272},
      YEAR = {2017},
     PAGES = {738--758},
      ISSN = {0022-1236,1096-0783},
   MRCLASS = {44A10 (52A20 52B45)},
       DOI = {10.1016/j.jfa.2016.09.011},
       URL = {https://doi.org/10.1016/j.jfa.2016.09.011},
}

@article{mouamine2025klain,
  author  = {M. A. Mouamine and F. Mussnig},
  title   = {A {K}lain--{S}chneider theorem for vector-valued valuations on convex functions},
  year    = {2026},
  journal = {J. Funct. Anal.},
  volume   = {291}, 
  pages ={Art. 111544},
  url     = {https://arxiv.org/abs/2503.07287}
}

@article {Mussnig2021,
    AUTHOR = {Mussnig, F.},
     TITLE = {Valuations on log-concave functions},
   JOURNAL = {J. Geom. Anal.},
  FJOURNAL = {Journal of Geometric Analysis},
    VOLUME = {31},
      YEAR = {2021},
     PAGES = {6427--6451},
      ISSN = {1050-6926,1559-002X},
   MRCLASS = {52A21 (26B25 46B20 52A41 52B45)},
       DOI = {10.1007/s12220-020-00539-3},
       URL = {https://doi.org/10.1007/s12220-020-00539-3},
}

@article {Mussnig2019,
    AUTHOR = {Mussnig, F.},
     TITLE = {Volume, polar volume and {E}uler characteristic for convex
              functions},
   JOURNAL = {Adv. Math.},
  FJOURNAL = {Advances in Mathematics},
    VOLUME = {344},
      YEAR = {2019},
     PAGES = {340--373},
      ISSN = {0001-8708,1090-2082},
   MRCLASS = {26B25 (46A40 52A20 52A41 52B45)},
       DOI = {10.1016/j.aim.2019.01.012},
       URL = {https://doi.org/10.1016/j.aim.2019.01.012},
}

@article {Mussnig2021b,
    AUTHOR = {Mussnig, F.},
     TITLE = {{${\rm SL}(n)$} invariant valuations on super-coercive convex
              functions},
   JOURNAL = {Canad. J. Math.},
  FJOURNAL = {Canadian Journal of Mathematics. Journal Canadien de
              Math\'ematiques},
    VOLUME = {73},
      YEAR = {2021},
     PAGES = {108--130},
      ISSN = {0008-414X,1496-4279},
   MRCLASS = {52A41 (26B25 46A40 52A20 52B45)},
       DOI = {10.4153/S0008414X19000531},
       URL = {https://doi.org/10.4153/S0008414X19000531},
}

@article {Schneider:75,
    AUTHOR = {Schneider, R.},
     TITLE = {Kinematische {B}er\"{u}hrma\ss e f\"{u}r konvexe {K}\"{o}rper},
   JOURNAL = {Abh. Math. Sem. Univ. Hamburg},
  FJOURNAL = {Abhandlungen aus dem Mathematischen Seminar der Universit\"{a}t
              Hamburg},
    VOLUME = {44},
      YEAR = {1975},
     PAGES = {12--23 (1976)},
      ISSN = {0025-5858},
   MRCLASS = {52A50},
       DOI = {10.1007/BF02992942},
       URL = {https://doi.org/10.1007/BF02992942},
}

@Book{Schneider,
  author    = {Schneider, R.},
  publisher = {Cambridge University Press, Cambridge},
  title     = {Convex {B}odies: the {B}runn-{M}inkowski {T}heory},
  year      = {2014},
  edition   = {{S}econd expanded},
  isbn      = {978-1-107-60101-7},
  series    = {Encyclopedia of Mathe\-matics and its Applications},
  volume    = {151},
  mrclass   = {52A39 (52-02 52A20)},
  pages     = {xxii+736},
}

@article {TrViJMAA,
    AUTHOR = {Tradacete, P. and Villanueva, I.},
     TITLE = {Radial continuous valuations on star bodies},
   JOURNAL = {J. Math. Anal. Appl.},
  FJOURNAL = {Journal of Mathematical Analysis and Applications},
    VOLUME = {454},
      YEAR = {2017},
     PAGES = {995--1018},
      ISSN = {0022-247X},
   MRCLASS = {52A30 (52B45)},
       DOI = {10.1016/j.jmaa.2017.05.026},
       URL = {https://doi.org/10.1016/j.jmaa.2017.05.026},
}

@article {TrViIntRep,
    AUTHOR = {Tradacete, P. and Villanueva, I.},
     TITLE = {Continuity and representation of valuations on star bodies},
   JOURNAL = {Adv. Math.},
  FJOURNAL = {Advances in Mathematics},
    VOLUME = {329},
      YEAR = {2018},
     PAGES = {361--391},
      ISSN = {0001-8708},
   MRCLASS = {52A30 (52B45)},
       DOI = {10.1016/j.aim.2018.02.021},
       URL = {https://doi.org/10.1016/j.aim.2018.02.021},
}

@article {TrViLattices,
    AUTHOR = {Tradacete, P. and Villanueva, I.},
     TITLE = {Valuations on {B}anach lattices},
   JOURNAL = {Int. Math. Res. Not. IMRN},
  FJOURNAL = {International Mathematics Research Notices. IMRN},
      YEAR = {2020},
     PAGES = {287--319},
      ISSN = {1073-7928},
   MRCLASS = {46B42 (28D05)},
       DOI = {10.1093/imrn/rny129},
       URL = {https://doi.org/10.1093/imrn/rny129},
}

@article {Tsang,
    AUTHOR = {Tsang, A.},
     TITLE = {Valuations on {$L^p$}-spaces},
   JOURNAL = {Int. Math. Res. Not. IMRN},
  FJOURNAL = {International Mathematics Research Notices. IMRN},
      YEAR = {2010},
     PAGES = {3993--4023},
      ISSN = {1073-7928,1687-0247},
   MRCLASS = {46E30 (46B40 46E05 52B45)},
       DOI = {10.1090/S0002-9947-2012-05681-9},
       URL = {https://doi.org/10.1090/S0002-9947-2012-05681-9},
}

@article {Vi,
    AUTHOR = {Villanueva, I.},
     TITLE = {Radial continuous rotation invariant valuations on star
              bodies},
   JOURNAL = {Adv. Math.},
  FJOURNAL = {Advances in Mathematics},
    VOLUME = {291},
      YEAR = {2016},
     PAGES = {961--981},
      ISSN = {0001-8708},
   MRCLASS = {52A30 (52B45)},
       DOI = {10.1016/j.aim.2015.12.030},
       URL = {https://doi.org/10.1016/j.aim.2015.12.030},
}

@article {Wang2014,
    AUTHOR = {Wang, T.},
     TITLE = {Semi-valuations on {${\rm BV}(\mathbb R^n)$}},
   JOURNAL = {Indiana Univ. Math. J.},
  FJOURNAL = {Indiana University Mathematics Journal},
    VOLUME = {63},
      YEAR = {2014},
     PAGES = {1447--1465},
      ISSN = {0022-2518,1943-5258},
   MRCLASS = {46E35 (52B45)},
       DOI = {10.1512/iumj.2014.63.5365},
       URL = {https://doi.org/10.1512/iumj.2014.63.5365},
}

@article {ZhuZhouXu,
    AUTHOR = {Zhu, B. and Zhou, J. and Xu, W.},
     TITLE = {Dual {O}rlicz--{B}runn--{M}inkowski theory},
   JOURNAL = {Adv. Math.},
  FJOURNAL = {Advances in Mathematics},
    VOLUME = {264},
      YEAR = {2014},
     PAGES = {700--725},
      ISSN = {0001-8708,1090-2082},
   MRCLASS = {52A40 (52A39 53A15)},
  MRNUMBER = {3250296},
MRREVIEWER = {Qingzhong\ Huang},
       DOI = {10.1016/j.aim.2014.07.019},
       URL = {https://doi.org/10.1016/j.aim.2014.07.019},
}

@article {DrewOrlicz,
    AUTHOR = {Drewnowski, L. and Orlicz, W.},
     TITLE = {On orthogonally additive functionals},
   JOURNAL = {Bull. Acad. Polon. Sci. S\'er. Sci. Math. Astronom. Phys.},
  FJOURNAL = {Bulletin de l'Acad\'emie Polonaise des Sciences. S\'erie des
              Sciences Math\'ematiques, Astronomiques et Physiques},
    VOLUME = {16},
      YEAR = {1968},
     PAGES = {883--888},
      ISSN = {0001-4117},
   MRCLASS = {46.35},
  MRNUMBER = {244755},
MRREVIEWER = {H.\ Gordon},
}

@article {Jonas2026,
    AUTHOR = {Knoerr, J.},
     TITLE = {Polynomial local functionals on convex functions},
   JOURNAL = {arXiv:2512.15402},
      YEAR = {2026},

}
\bibliographystyle{acm}

\end{document}